\documentclass{iejax}
\usepackage{amsmath,amsthm,amssymb,amsfonts,amscd, array}
\usepackage[all]{xy}

\newtheorem{thm}{Theorem}[section]
 \newtheorem{cor}[thm]{Corollary}
 \newtheorem{lemma}[thm]{Lemma}
 \newtheorem{prop}[thm]{Proposition}
\theoremstyle{definition}
 
\newtheorem{rem}[thm]{Remark}
 \newtheorem{example}[thm]{Example}

\usepackage{geometry}                		
\usepackage{graphicx}				
\usepackage{rotating}

\def\ord{{\mathrm{ord}}}

\def\J{{\mathrm{J}}}
\def\LL{{\mathrm{LL}}}
\def\lcm{{\mathrm{lcm}}}
\def\GAP{{\sf GAP}}

\def\N{{\mathbb{N}}}
\def\Q{{\mathbb{Q}}}
\def\Z{{\mathbb{Z}}}

\usepackage{xstring}
\newcommand\refPartI[1]{%
\IfStrEqCase{#1}{{thm_upper_and_lower_bounds_I}{7.1}
    {lem_LL_elementary}{7.1}
    {cor_n_leq_3}{7.1}
    {prop_lambda}{3.2}
    {prop_m_and_e1}{6.2}  
    {rem_bound_overview}{7.1}
    {e_equals2346}{7.1}  
    {thm_reduction_of_n}{5.1}
    {cor_uniserial_case}{5.1}
    {isomorphism_by_Loewy_vector}{5.2}  
    {m_and_digit_sums}{6.1}  
    {example_m_trivial}{6.1}
    {same_m_for_same_subgroup}{6.1}
    {lem_m_eq_2}{6.3}
    {rem_m_eq_2_cyclic_case}{6.1}
    {lem_m_bounded_by_p}{6.4}}  
    [???]}

\title{The Loewy Structure of Certain Fixpoint Algebras, Part II}
\author{T. Breuer, L. H\'ethelyi, E. Horv\'ath and B. K\"ulshammer}

\address{
{\bf T. Breuer} \\
Lehrstuhl D f\"ur Mathematik \\
RWTH Aachen University \\
Pontdriesch 14-16 \\
D-52062 Aachen, Germany \\
e-mail: sam@math.rwth-aachen.de \\
\mbox{} \\
{\bf L. H\'ethelyi} \\
Department of Algebra \\
Budapest University of Technology and Economics \\
M\H uegyetem rkp. 3-9 \\
H-1111 Budapest, Hungary \\
e-mail: fobaba@t-online.hu \\
\mbox{} \\
E. Horv\'ath \\
Department of Algebra \\
Budapest University of Technology and Economics \\
M\H uegyetem rkp. 3-9 \\
H-1111 Budapest, Hungary \\
e-mail: he@math.bme.hu \\
\mbox{} \\
B. K\"ulshammer \\
Institut f\"ur Mathematik \\
Friedrich-Schiller-Universit\"at \\
D-07737 Jena, Germany \\
e-mail: kuelshammer@uni-jena.de}

\thanks{The work on this paper was begun while the last author was visiting
the Technical University of Budapest in October 2017.
Part of the work was also done while the last author was in residence
at the Mathematical Sciences Research Institute in Berkeley, California,
during the Spring 2018 semester,
supported by the National Science Foundation under Grant No.~DMS-1440140.
The first author gratefully acknowledges support by the German Research
Foundation (DFG) within the SFB-TRR 195 {\it Symbolic Tools in Mathematics
and their Applications}.
The research in this paper was also supported by the NKFI-Grants
No.~115288 and 115799.}

\begin{document}

\begin{abstract}
In Part I of this paper, we introduced a class of certain algebras of finite dimension over a field. 
All these algebras are split, symmetric and local. Here we continue to investigate their Loewy structure. 
We show that in many cases their Loewy length is equal to an upper bound established in Part I, 
but we also construct examples where we have a strict inequality.
 \\[+2mm]
\phantom{ \subjclassname{???}} \\
\phantom{ {\bf Keywords}: ???}
\end{abstract}

\maketitle


\section{Introduction}

In Part I of this paper, we introduced a class of finite-dimensional algebras $A(q,n,e)$ over a field $F$, depending on parameters $q,n,e \in \N$ such that $q > 1$ and 
$e \mid q^n-1$. All these algebras are split, local and symmetric; their dimension is $z+1$ where $z := \frac{q^n-1}{e}$. When $q$ is a prime $p$ and $F$ is algebraically closed of
characteristic $p$ then $A(q,n,e)$ is isomorphic to a fixpoint algebra $(FP)^H$ where $FP$ is the group algebra of an elementary abelian $p$-group $P$ of order $p^n$ over $F$  and 
$H$ is a cyclic group of order $e$ acting freely on $P \setminus \{1\}$.

In Part I, we presented an inductive procedure in order to compute the Loewy structure of $A(q,n,e)$, and we proved the following upper bound for the Loewy length of $A(q,n,e)$:
\begin{equation}
 \LL(A(q,n,e)) \leq \left\lfloor n \frac{q-1}{m} \right\rfloor + 1;
\end{equation}
here $m = m(q,e)$ is defined as the smallest positive integer $t$ such that there exists a sum of $t$ powers of $q$ which is divisible by $e$. This number can be defined for 
arbitrary $q,e \in \N$ with $\gcd(q,e) = 1$. 

In \cite[Corollary~\refPartI{cor_uniserial_case}]{BHHK1} we proved that $A(q,n,e)$ is uniserial if and only if $e$ is divisible by $\frac{q^n-1}{q-1}$. 

The main purpose of this paper is to show that, in many cases, the inequality in (1) is in fact an equality. For example, we will show that
\begin{equation}
 \LL(A(q,n,e)) = \left\lfloor n \frac{q-1}{m} \right\rfloor + 1
\end{equation}
whenever one of the following conditions is satisfied:
\begin{itemize} 
 \item
   $n \leq 3$,
   see~\cite[Corollary~\refPartI{cor_n_leq_3}]{BHHK1}.
   (There are examples for $n=5$ where (2) does not hold,
   see Remark~\ref{ref_database_examples};
   the case $n=4$ is still open. For $n=4$ and $q \leq 100$, the equality (2) holds.)
 \item
   $e \leq 32$,
   see Proposition~\ref{prop_e_at_most_32}.
   (There are examples for $e = 33$ where (2) does not hold,
   see Proposition~\ref{prop_small_LL_e=33}.)
 \item
   $e \mid \frac{q^d-1}{q-1}$ for $d \in \{1,\ldots,5\}$,
   see Propositions~\ref{prop_e_divides_Phi_d} and~\ref{prop_e_divides_Phi_5}.
   (There are examples for larger $d$ where (2) does not hold,
   see Remark~\ref{ref_database_examples}.)
 \item
   $e \mid \Phi_d(q)$ where $d$ is a power of $2$ or $d \in \{3,5,6,9,10\}$,
   see Remark~\ref{rem_e_divides_Phi_d}.
   (Here $\Phi_d(X)$ denotes the $d$-th cyclotomic polynomial.) 
 \item
   $e$ is a power of a Pierpont prime,
   see Theorem~\ref{thm_ll_for_prime_powers}.
   (A prime number $p$ is called a Pierpont prime
   if it has the form $p = 1+2^a3^b$ where $a,b \in \N_0$.)
\item
   $e$ is a divisor of $q^n-1$ and a multiple of $q^{\frac{n}{2}}-1$, see Lemma~\ref{lemma_loewy_length_three}.
 \item
   $m(q,e) \mid q-1$, see \cite[Theorem~\refPartI{thm_upper_and_lower_bounds_I}]{BHHK1}.
 \item
   $m(q,e) = 2$, see \cite[Lemma~\refPartI{lem_m_eq_2}]{BHHK1}.
 \item
   $m(q,e) \ge \frac{e}{3}$, see Proposition~\ref{prop_LL_for_large_m_relative_to_e}.
 \item
   $z < 70$.
   (There is an example for $z = 70$ where (2) does not hold,
   see Example~\ref{expl_z_eq_70}.)
\end{itemize}
More conditions and details can be found in the body of this paper.

We also define a certain equivalence relation on our set of algebras. Two equivalent algebras are isomorphic, and the upper bounds for their Loewy length in (1) are the same. 
One of us (T. B.) has computed a database with 768\,511 equivalence classes
of algebras.
These contain all algebras $A(q,n,e)$ with $z \leq 10\,000$.
These algebras fall into at least $478\,145$
and at most $484\,234$ isomorphism classes,
see Remark~\ref{ref_database_examples}.

It turns out that the 
equality (2) holds for 757 790 of these equivalence classes. In only one of the equivalence classes the difference between both sides in (1) is bigger than 1 (namely 2). 
Thus, at least for algebras of dimension up to 10 000, the bound for the Loewy length in (1) appears to be reasonable. 

In this paper, we will denote by $\J(A)$ the Jacobson radical and by $\LL(A)$ the Loewy length of a finite-dimensional algebra $A$. 
If $\LL(A) = l$ and $\dim \J(A)^{i-1}/\J(A)^i = d_i$ for $i = 1,\ldots, l$ then $(d_1,\ldots,d_l)$ is called the Loewy vector of $A$. 

For $q,n,e \in \N$ with $q>1$ and $e \mid 
q^n-1$, the $F$-algebra $A(q,n,e)$ is constructed as follows. Consider the ideal $I := (X_1^q,\ldots,X_n^q)$ of the polynomial algebra $F[X_1,\ldots,X_n]$, and set $x_j := 
X_j+I$ for $j=1,\ldots,n$. Then $A(q,n,e)$ is the subalgebra of $F[X_1,\ldots,X_n]/I = F[x_1,\ldots,x_n]$ generated by all monomials $x_1^{i_1} \ldots x_n^{i_n}$ such that
$i_1 + qi_2 + \ldots + q^{n-1}i_n \equiv 0 \pmod{e}$; note that $x_j^q = 0$ for $j=1,\ldots,n$. We showed in Part I that the elements $b_0,b_1,\ldots,b_z$ constitute an
$F$-basis of $A(q,n,e)$ where  $b_k = x_1^{i_1} \ldots x_n^{i_n}$ and $ke = i_1 + qi_2 + \ldots + q^{n-1}i_n$ is the $q$-adic expansion of $ke$, for $k=0,\ldots,z$. 

Our paper is structured as follows. In Section 2, we deal with the function $m(q,e)$ and prove several properties. Tables with the values of this function can be found at the 
end of the paper. In Section 3 we present various methods in order to obtain lower bounds for the Loewy length of $A(q,n,e)$. In Section 4 we investigate the validity of (2) in
the situation where $e$ is a prime power, and in Section 5 we consider the case where $e$ divides $\frac{q^n-1}{q-1}$. Section 6 contains our results for the case when $e$ is a 
small number. Here we also present a series of examples where the inequality in (1) is strict. Then we deal with algebras where $m(q,e)$ is large (relative to $e$) or $q$ is 
small. In the last part of the paper, we change our perspective and consider all the algebras $A(q,n,e)$ of a fixed dimension $d = z+1$.


\section{Congruence Properties of Sums of Powers} 

Let $q,e \in \N$ such that $\gcd(q,e) = 1$.
In \cite[Section 6]{BHHK1}, we defined $m(q,e)$ as the smallest positive
integer $t$ with the property that there exists a sum of $t$ powers
of $q$ which is divisible by $e$. Then $1 \leq m(q,e) \leq e$; moreover,
$m(q,e) = e$ if and only if $q \equiv 1 \pmod{e}$, and $m(q,e) = 1$ if and only if $e = 1$ (cf. \cite[Example~\refPartI{example_m_trivial}]{BHHK1}).

For $q > 1$, there is a slightly different description of $m(q,e)$. Recall that we denote by $s_q(x) = x_0 + x_1 + \ldots + x_n$ the $q$-adic digit sum of a nonnegative integer $x$
with $q$-adic expansion $x = x_0 + x_1q + \ldots + x_nq^n$. In \cite[Proposition~\refPartI{m_and_digit_sums}]{BHHK1} we proved that
$$m(q,e) = \min\{s_q(ke): k \in \N\} = \min\{s_q(ke): k = 1,\ldots,z\}$$
where $z := \frac{q^n-1}{e}$ for any $n \in \N$ such that $e \mid q^n-1$; usually we take $n := \mathrm{ord}_e(q)$ where $\mathrm{ord}_e(q)$ denotes the order of $q+e\Z$ in $(\Z/e\Z)^\times$,
the order of $q$ modulo $e$. 

Our first result is related to \cite[Proposition~\refPartI{thm_reduction_of_n}]{BHHK1}.

\begin{prop}\label{m_functional_equation}
Let $q, n, e \in \N$ such that $q > 1$ and $e \mid q^n-1$.
Moreover, let $n', e' \in \N$ such that $n' \mid n$ and 
$e = e' \frac{q^n-1}{q^{n'}-1}$.
Then $m(q,e) = \frac{n}{n'} m(q,e')$.
\end{prop}

\begin{proof}
We set $z:= \frac{q^n-1}{e} = \frac{q^{n'}-1}{e'}$.
Since 
$m(q,e) = \min \{s_q(ke): k = 1,\ldots,z \}$ 
and
$m(q,e') = \min \{s_q(ke'): k = 1,\ldots,z\}$ it suffices to show that $s_q(ke) = \frac{n}{n'} s_q(ke')$ for $k=1,\ldots,z$.
Let $k \in \{1,\ldots,z\}$, and consider the $q$-adic expansion $ke' = \sum_{t=1}^{n'} q^{t-1}i_t$. Then 
\begin{eqnarray*}
 ke &= ke'\frac{q^n-1}{q^{n'}-1} = (\sum_{t=1}^{n'} q^{t-1}i_t)(1 + q^{n'} + q^{2n'} + \ldots + q^{n-n'}) \cr
 &= \sum_{t=1}^{n'}q^{t-1}i_t + \sum_{t=1}^{n'} q^{n'+t-1}i_t + \ldots + \sum_{t=1}^{n'} q^{n-n'+t-1}i_t 
\end{eqnarray*} 
is the $q$-adic expansion of $ke$. Hence $s_q(ke) = \frac{n}{n'} \sum_{t=1}^{n'} i_t =  \frac{n}{n'} s_q(ke')$.
\end{proof}

We record two special cases.

\begin{cor}\label{cor_m_functional_equation}
Let $q, n, e \in \N$ such that $q > 1$ and $e \mid q^n-1$.
\begin{itemize}
\item[(i)]
    If $e' \in \N$ such that $e = e' \frac{q^n-1}{q-1}$ then $m(q,e) = ne'$.
\item[(ii)]
    If $n' \in \N$ such that $n' \mid n$ and $e = \frac{q^n-1}{q^{n'}-1}$
    then $m(q,e) = \frac{n}{n'}$.
\end{itemize}
\end{cor}

\begin{proof}
Apply Proposition~\ref{m_functional_equation} with $n' = 1$ or $e' = 1$,
respectively.
In part~(i), we get $m(q,e) = n m(q,e')$,
and since $\frac{q-1}{e'} = \frac{q^n-1}{e} \in \N$,
\cite[Example~\refPartI{example_m_trivial}~(i)]{BHHK1} implies
that $m(q,e') = e'$.
In part~(ii), we get $m(q,e) = \frac{n}{n'} m(q,1)$,
and \cite[Example~\refPartI{example_m_trivial}~(ii)]{BHHK1} implies
that $m(q,1) = 1$. 
\end{proof}

\begin{example}
Table~\ref{table_meq_small} gives some of the numbers $m(q,e)$.
The columns in this table are labelled by $e$, and the rows by $q$.
The computations were done
using the computer algebra system {\GAP}~\cite{GAP}.
\end{example}

\begin{example}\label{example_m_eq_2}
Let $q, e \in \N $ such that $1 \not\equiv q \equiv -1 \pmod{e}$.
Then $m(q,e) = 2$; in fact, 
$m(q,e) \leq 2$ since $e \mid 1+q$, and $m(q,e) \neq 1$ since $ e \neq 1$.
\end{example}

Next we investigate the situation where $m(q,e)$ is large.

\begin{prop}\label{prop_large_m_relative_to_e}
Let $q, e \in \N$ such that $1 < q < e$ and $\gcd(q,e) = 1$.
Then $m := m(q,e) \geq \frac{e}{3}$ if and only if one of the following holds:
\begin{itemize}
\item[(i)]
    $q \geq 3$, $\gcd(2,q) = 1$, $e = 2q-2$ (where $m = \frac{e}{2} = q-1$).
\item[(ii)]
    $q \geq 4$, $\gcd(3,q) = 1$, $e = 3q-3$ (where $m = \frac{e}{3} = q-1$).
\item[(iii)]
    $q \geq 5$, $\gcd(6,q) = 1$, $e = \frac{3q-3}{2}$
    (where $m = \frac{e}{3} = \frac{q-1}{2}$).
\item[(iv)]
    the pair $(q,e)$ appears in the following table:
\begin{center}
\begin{tabular}{c|ccc|cc|ccc|c}
$q$ & $2$ & $2$ & $2$ & $3$ & $3$ & $4$ & $4$ & $4$ & $5$ \cr
$e$ & $3$ & $5$ & $7$ & $5$ & $8$ & $5$ & $7$ & $15$ & $24$ \cr \hline
$m$ & $2$ & $2$ & $3$ & $2$ & $4$ & $2$ & $3$ & $6$ & $8$
\end{tabular}
\end{center}
\end{itemize}
\end{prop}

\begin{proof}
For the ``if'' direction,
observe that the claimed values of $m$ for given $q$ and $e$
are equal to both $e_1 = \gcd(e, q-1)$ and $s_q(e)$ in the cases (i)--(iii), 
so that we can apply \cite[Lemma~\refPartI{prop_m_and_e1}~(i)]{BHHK1}
and \cite[Proposition~\refPartI{m_and_digit_sums}]{BHHK1}.
In case (iv), the values for $m$ are given in Table~\ref{table_meq_small}. In each case, we have $3m(q,e) \geq e$.

For the ``only if'' direction, suppose that $m \geq e/3$. 
A careful inspection of Table~\ref{table_meq_small} shows that we may assume $e > 30$.
Consider the $q$-adic expansion
$e = i_1 + qi_2 + \ldots + q^li_{l+1}$
where $i_{l+1} \neq 0$.
Then $m \leq s_q(e) = i_1 + i_2 + \ldots + i_{l+1}$,
so $0 \leq 3 m - e \leq 2 i_1 + (3-q) i_2 + \cdots + (3-q^l) i_{l+1}$.

If $l \geq 2$
then $0 \leq 3m-e \leq 2i_1 + (3-q)i_2 + (3-q^l)i_{l+1}$.
Assume first that $q \geq 3$.
Then 
\begin{eqnarray*}
   0 & \leq & 2i_1 + (3-q^l)i_{l+1} \leq 2(q-1) + (3-q^l) = 2q+1-q^l \\
     & = & 1 - q \left( q^{l-1} - 2 \right) \leq 1 - 3 (3-2) < 0,
\end{eqnarray*}
a contradiction.
Thus we must have $q = 2$.
Then $0 \leq 2i_1 + i_2 + (3-2^l) \leq 6 - 2^l$.
This implies that $l=2$, and we have the contradiction $e = i_1 +2i_2 + 4i_3 \leq 7$.

{}From now on, assume $l = 1$.
Then $0 \leq 3m-e \leq 2i_1 + (3-q)i_2$,
i.~e., $(q-3)i_2 \leq 2i_1 \leq 2(q-1)$.
Since $30 < e = i_1 + qi_2 < q^2$ this implies that $q \geq 6$.

\begin{itemize}
\item[$i_2 \geq 3$:]
    Then $3(q-3) \leq 2i_1 \leq 2(q-1)$, i.~e., $q \leq 7$.

    If $q = 6$ then $9 \leq 3i_2 \leq 2i_1 \leq 10$,
    i.~e., $i_2 = 3$ and $i_1 = 5$.
    Thus we have the contradiction $e = 23$.

    If $q = 7$ then $12 \leq 4i_2 \leq 2i_1 \leq 12$,
    i.~e., $i_2 = 3$ and $i_1 = 6$.
    Hence we have the contradiction $e = 27$.

\item[$i_2 = 2$:]
    Then $q-3 \leq i_1 \leq q-1$.

    If $i_1 = q-1$ then $e = 2q + (q-1) = 3q-1$.
    Since $(q-1)(3q-1) = 3q^2-4q+1 = 2q^2+q(q-4) + 1$, this
    leads to the contradiction $m \leq q-1 < \frac{e}{3}$.

    If $i_1 = q-2$ then $e = 2q + (q-2) = 3q-2$.
    Since $(q-1)(3q-2) = 3q^2 -5q + 2 = 2q^2 + q (q-5) + 2$, this
    leads to the contradiction $m \leq q-1 < \frac{e}{3}$.

    If $i_1 = q-3$ then $e = 2q + (q-3) = 3q-3$ and $m \leq q-1 = \frac{e}{3}$, i.e. $m = \frac{e}{3} = q-1$.
    Since $\gcd( q, e ) = 1$, we must have $3 \nmid q$.
    Thus we are in case (ii).

\item[$i_2 = 1$:]
    Then $30 < e = q + i_1$ and $\frac{q-3}{2} \leq i_1 < q$, so that $q \geq 16$.

    If $i_1 = \frac{q-3}{2}$ then $q$ is odd, $e = q + \frac{q-3}{2} = \frac{3q-3}{2}$ and $m \leq \frac{q-1}{2} = \frac{e}{3}$, i.e. $m = \frac{e}{3} = \frac{q-1}{2}$.
    Since $\gcd(q,e) = 1$ we also have $3 \nmid q$, so that $\gcd(q,6) = 1$.
    Thus we are in case (iii).

    If $i_1 = \frac{q-2}{2}$ then $q$ is even and $e = \frac{3q-2}{2}$.
    Since $\gcd(q,e) = 1$ this implies $4 \mid q$.
    Then $e^2 = 2q^2 + \frac{q-12}{4} q + 1$.
    This leads to the contradiction
    $m \leq 2 + \frac{q-12}{4} + 1 = \frac{q}{4} < \frac{e}{3}$.

    If $i_1 = \frac{q-1}{2}$ then $q$ is odd and $30 < e = \frac{3q-1}{2}$, i.e. $q > 20$.
    If $q \equiv 1 \pmod{4}$ then
    $\frac{3q-3}{2} e = 2q^2 + \frac{q-13}{4}q + \frac{q+3}{4}$.
    This leads to the contradiction
    $m \leq 2 + \frac{q-13}{4} + \frac{q+3}{4} = \frac{q-1}{2} < \frac{e}{3}$.
    If $q \equiv 3 \pmod{4}$ then
    $\frac{3q-5}{2} e = 2q^2 + \frac{q-19}{4} q + \frac{q+5}{4}$.
    This leads to the contradiction
    $m \leq 2 + \frac{q-19}{4} + \frac{q+5}{4} = \frac{q-3}{2} < \frac{e}{3}$.

    If $i_1 = \frac{q}{2}$ then $q$ is even and $e = \frac{3q}{2}$.
    Since $\gcd(q,e) = 1$ this implies $4 \nmid q$.
    Then $\frac{3q-2}{2} e = 2q^2 + \frac{q-6}{4} q$.
    This leads to the contradiction
    $m \leq 2 + \frac{q-6}{4} = \frac{q+2}{4} < \frac{e}{3}$.

    If $\frac{q+1}{2} \leq i_1 \leq q-5$ then set $x:= q-i_1$
    and write $q = ax+k$ where $a, k \in \mathbb{N}_0$ and $0 \leq k < x$.
    Then $5 \leq x \leq \frac{q-1}{2}$ and $a \geq 2$.
    Moreover, we have $ae = 2aq-ax = k+(2a-1)q$.
    Thus $m \leq s_q(ae) \leq k + (2a-1)$.
    This implies:
    \[
       3k+6a-3 \geq 3m \geq e = 2q-x = 2(ax+k)-x = 2k+(2a-1)x.
    \]
    Hence $6a-3 \geq (2a-1)x - k > (2a-2) x \geq 10a-10$,
    and we have the contradiction $4a<7$.

    If $i_1 = q-4$ then $e = 2q-4$. Since $\gcd(q,e) =1$ this implies that $q$ is odd.   
    If $q \equiv 0 \pmod{3}$ then
    $\frac{2q}{3} e = q^2 + \frac{q-9}{3} q + \frac{q}{3}$.
    This leads to the contradiction
    $m \leq 1 + \frac{q-9}{3} + \frac{q}{3} = \frac{2q-6}{3} < \frac{e}{3}$.
    If $q \equiv 1 \pmod{3}$ then
    $\frac{2q+1}{3} e = q^2 + \frac{q-7}{3} q + \frac{q-4}{3}$.
    This leads to the contradiction
    $m \leq 1 + \frac{q-7}{3} + \frac{q-4}{3} = \frac{2q-8}{3} < \frac{e}{3}$.
    If $q \equiv 2 \pmod{3}$ then
    $\frac{2q+2}{3} e = q^2 + \frac{q-5}{3} q  + \frac{q-8}{3}$.
    This leads to the contradiction
    $m \leq 1 + \frac{q-5}{3} + \frac{q-8}{3} = \frac{2q-10}{3} < \frac{e}{3}$.

    If $i_1 = q-3$ then $e = 2q-3$.
    Since $\gcd(q,e) = 1$ we conclude that $q \not\equiv 0 \pmod{3}$.
    If $q \equiv 1 \pmod{3}$ then
    $\frac{2q-2}{3} e = q^2 + \frac{q-10}{3} q + 2$.
    This leads to the contradiction
    $m \leq 1 + \frac{q-10}{3} + 2 = \frac{q-1}{3} < \frac{e}{3}$.
    If $q \equiv 2 \pmod{3}$ then
    $\frac{2q-1}{3} e = q^2 + \frac{q-8}{3}q + 1$.
    This leads to the contradiction
    $m \leq 1 + \frac{q-8}{3} + 1 = \frac{q-2}{3} < \frac{e}{3}$.

    If $i_1 = q-2$ then $e = 2q-2$ and $m \leq 1 + (q-2) = q-1$.
    Since $\gcd(q,e) = 1$ we conclude that $q$ is odd.
    Thus $q-1 = \gcd(e,q-1) \mid m$, so that $m = q-1$.
    Thus we are in case (i).

    If $i_1 = q-1$ then $e = 2q-1$.
    If $q \equiv 0 \pmod{3}$ then
    $\frac{2q-3}{3} e = q^2 + (\frac{q}{3} - 3) q + (\frac{q}{3} + 1)$.
    This leads to the contradiction
    $m \leq 1 + (\frac{q}{3} -3) + (\frac{q}{3} +1) =
    \frac{2q-3}{3} < \frac{e}{3}$.
    If $q \equiv 1 \pmod{3}$ then
    $\frac{2q-2}{3} e = q^2 + \frac{q-7}{3} q + \frac{q+2}{3}$.
    This leads to the contradiction
    $m \leq 1 + \frac{q-7}{3} + \frac{q+2}{3} = \frac{2q-2}{3} < \frac{e}{3}$.
    If $q \equiv 2 \pmod{3}$ then
    $\frac{2q-4}{3} e = q^2 + \frac{q-11}{3} q + \frac{q+4}{3}$.
    This leads to the contradiction
    $m \leq 1 + \frac{q-11}{3} + \frac{q+4}{3} = \frac{2q-4}{3} < \frac{e}{3}$.
\end{itemize}
This finishes the proof.
\end{proof}

Next we drop the assumption $q<e$ from Proposition~\ref{prop_large_m_relative_to_e}. 
If $q \equiv 1 \pmod{e}$ then $m(q,e) = e \geq \frac{e}{3}$ by
\cite[Example~\refPartI{example_m_trivial}]{BHHK1}.
Thus we can and will ignore this case.

\begin{prop}\label{q_bigger_e}
 Let $q,e \in \mathbb{N}$ such that $\gcd(q,e) = 1 \not\equiv q \pmod{e}$. Then $m := m(q,e) \geq \frac{e}{3}$ if and only if
one of the following holds:

\begin{itemize}
\item[(i)]
    $q \geq 3$, $\gcd(2,q) = 1$,
    $e = \frac{2q-2}{k}$ where $k$ is an odd divisor of $q-1$
    (where $m = \frac{e}{2}$).
\item[(ii)]
    $q \geq 4$, $\gcd(3,q) = 1$,
    $e = \frac{3q-3}{k}$ where $k$ is a divisor of $q-1$
    with $k \equiv 1 \pmod{3}$
    (where $m = \frac{e}{3}$).
\item[(iii)]
    $q \geq 5$, $\gcd(3,q) = 1$,
    $e = \frac{3q-3}{k}$ where $k$ is a divisor of $q-1$
    with $k \equiv 2 \pmod{3}$
    (where $m = \frac{e}{3}$).
\item[(iv)]
    $q \equiv b \pmod{e}$, and the pair $(b,e)$ appears in the following table:
\begin{center}
\begin{tabular}{c|ccc|cc|ccc|c}
$b$ & $2$ & $2$ & $2$ & $3$ & $3$ & $4$ & $4$ & $4$ & $5$ \cr
$e$ & $3$ & $5$ & $7$ & $5$ & $8$ & $5$ & $7$ & $15$ & $24$ \cr \hline
$m$ & $2$ & $2$ & $3$ & $2$ & $4$ & $2$ & $3$ & $6$ & $8$
\end{tabular}
\end{center}
\end{itemize}
\end{prop}

\begin{proof}
Suppose first that $m := m(q,e) \geq \frac{e}{3}$,
and write $q = ae+b$ where $a,b \in \N_0$ and $0 \leq b < e$.
If $b=0$ then $1 = \gcd(q,e) = e$ which is impossible since
$q \not\equiv 1 \pmod{e}$.
Thus $1 < b < e$.
Since $\gcd(b,e) = \gcd(q,e) = 1$ and $m(b,e) = m(q,e) \geq \frac{e}{3}$ 
Proposition~\ref{prop_large_m_relative_to_e} applies, with
$b$ instead of $q$. We discuss the various cases.

(i) Let $b \geq 3$, $\gcd(2,b) = 1$, $e = 2b-2$, $m = \frac{e}{2} = b-1$. Then $q \geq b \geq 3$, and $\gcd(2,q) = 
\gcd(2,b) = 1$ since $2 \mid e$. Moreover, $e = 2(q-ae)-2$, so that $e(1+2a) = 2q-2$ and $e = \frac{2q-2}{1+2a}$. Thus we
are in case (i) of Proposition~\ref{q_bigger_e}.

(ii) Let $b \geq 4$, $\gcd(3,b) = 1$, $e = 3b-3$, $m = \frac{e}{3} = b-1$. Then $q \geq b \geq 4$, and $\gcd(3,q) = 
\gcd(3,b) = 1$ since $3 \mid e$. Moreover, $e = 3(q-ae)-3$, so that $e(1+3a) = 3q-3$ and $e = \frac{3q-3}{1+3a}$. Thus we
are in case (ii) of Proposition~\ref{q_bigger_e}.

(iii) Let $b \geq 5$, $\gcd(6,b) = 1$, $e = \frac{3b-3}{2}$, $m = \frac{e}{3} = \frac{b-1}{2}$. Then $q \geq b \geq 5$,
and $\gcd(3,q) = \gcd(3,b) = 1$ since $3 \mid e$. Moreover, $2e = 3(q-ae)-3$, so that $e(2+3a) = 3q-3$ and $e = 
\frac{3q-3}{2+3a}$. Thus we are in case (iii) of Proposition~\ref{q_bigger_e}.

(iv) If one of these cases holds for $b$ then we are clearly in the corresponding case of Proposition~\ref{q_bigger_e}.

Now suppose, conversely, that we are in one of the cases of Proposition~\ref{q_bigger_e}.

(i) Let $q \geq 3$, $\gcd(2,q) = 1$ and $e = \frac{2q-2}{k}$ where $k$ is an odd divisor of $q-1$. We write $k = 2a+1$
with $a \in \mathbb{N}_0$. Then $e(2a+1) = 2q-2$, so that $e+2 = 2(q-ae)$ and $b := q-ae = \frac{e+2}{2} \in \mathbb{N}$.
If $e=2$ then we obtain the contradiction $k = q-1 \equiv 0 \pmod{2}$. Thus $e \geq4$ and $3 \leq b < e$. Moreover,
$\gcd(b,e) = \gcd(q,e) = 1$. Since $2 \mid e$ this implies $\gcd(2,b) = 1$. Hence we are in case (i) of Proposition~\ref{prop_large_m_relative_to_e},
with $b$ instead of $q$. In particular, $m(q,e) = m(b,e) = \frac{e}{2} \geq \frac{e}{3}$. 

(ii) Let $q \geq 4$, $\gcd(3,q) = 1$ and $e = \frac{3q-3}{k}$ where $k$ is a divisor of $q-1$ with $k \equiv 1 \pmod{3}$.
We write $k = 3a+1$ with $a \in \mathbb{N}_0$. Then $e(3a+1) = 3q-3$, so that $e+3 = 3(q-ae)$ and $b := q-ae = \frac{e+3}{3}
\in \mathbb{N}$; in particular, $\gcd(b,e) = \gcd(q,e) = 1$. Since $3 \mid e$ this implies $\gcd(3,b) = 1$. Obviously,
$1 < \frac{e}{3}+1<e$.

If $b=2$ then $e=3$. Thus we are in case (iv) of Proposition~\ref{prop_large_m_relative_to_e}, with $b$ instead of $q$. Hence $m(q,e) = m(b,e) = 2 \geq
\frac{e}{3}$.

Thus we may assume $b \geq 4$. Then we are in case (ii) of Proposition~\ref{prop_large_m_relative_to_e}, with $b$ instead of $q$. Hence $m(q,e) =
m(b,e) = \frac{e}{3}$.

(iii) Let $q \geq 5$, $\gcd(3,q) = 1$ and $e = \frac{3q-3}{k}$ where $k$ is a divisor of $q-1$ with $k \equiv 2 \pmod{3}$.
We write $k = 3a+2$ with $a \in \mathbb{N}_0$. Then $e(3a+2) = 3q-3$, so that $2e+3 = 3(q-ae)$ and $b := q-ae =
\frac{2e+3}{3} \in \mathbb{N}$; in particular, $\gcd(b,e) = \gcd(q,e) = 1$. Since $3 \mid e$ this implies $\gcd(3,b) = 1$.
Since $b \equiv 3b = 2e+3 \equiv 1 \pmod{2}$ we even have $\gcd(6,b) = 1$. Since $b = \frac{2e}{3} + 1 > 1$ this implies
$b \geq 5$ and $e \geq 6$. Thus $b = \frac{2e+3}{3} < e$. Hence we are in case (iii) of Proposition~\ref{prop_large_m_relative_to_e}, with $b$ instead
of $q$. Hence $m(q,e) = m(b,e) = \frac{e}{3}$. 

(iv) Let $q \equiv 2 \pmod{e}$ where $e=3$. Then $m(q,e) = m(2,3) = 2 \geq \frac{e}{3}$. 

The other cases are similar.
\end{proof}

Next we investigate the situation in the case where $\ord_e(q) = 2$. 
(The case $\ord_e(q) = 1$, i.e. $q \equiv 1 \pmod{e}$,
is part of \cite[Example~\refPartI{example_m_trivial}]{BHHK1}.)

\begin{prop}\label{m_eq_e1}
Let $q,e \in \N$ such that $q>1$ and $e \mid q^2-1$, and set
$e_1 := \gcd(e,q-1)$, $e_2 := \gcd(e,q+1), m := m(q,e)$.
If $e_1 \geq e_2$ or both $e$ and $\frac{q^2-1}{e}$ are even then $m=e_1$.
Otherwise $m=2e_1$.
\end{prop}

\begin{proof}
\cite[Lemma~\refPartI{prop_m_and_e1}]{BHHK1} implies that $m \in
\{e_1,2e_1\}$.

Suppose first that both $e$ and $\frac{q^2-1}{e}$ are even. 
Then $q$ is odd, $e_1$ is even, and $l := \gcd(2,e_1,\frac{q^2-1}{e}) = 2$. 
Thus $m \leq \frac{2e_1}{l} = e_1$ by \cite[Lemma~\refPartI{prop_m_and_e1}~(iii)]{BHHK1},
i.e.\ $m = e_1$ by \cite[Lemma~\refPartI{prop_m_and_e1}~(i)]{BHHK1}.

Suppose next that $e_1 = e_2$. Then $e_1 \mid \gcd(q-1,q+1) \mid 2$, i.e. $e_1
= e_2 \in \{1,2\}$.
Thus $e = e_1 \mid q-1$, and $m = e = e_1$ by
\cite[Example~\refPartI{example_m_trivial}]{BHHK1}.

Now suppose that $e_1 > e_2$. If $e$ is odd then
there is $a \in \{0,\ldots,e_2-1\}$ such that $2a \equiv e_1 \pmod{e_2}$.
Since $e=e_1e_2$ one checks easily that $a+q(e_1-a) \equiv 0 \pmod{e}$. Thus
$m \leq e_1$, i.e.\ $m=e_1$.

If $e$ is even then $q$ is odd, and both $e_1$ and $e_2$ are even. Then
$\frac{e_1 \pm e_2}{2} \in \mathbb{N}$.
We claim that $e$ divides $\frac{e_1+e_2}{2} + q\frac{e_1-e_2}{2} = e_1
\frac{q+1}{2} - e_2 \frac{q-1}{2}$.
(Then $m \leq e_1$, i.e. $m = e_1$.)
We write $e_1 = a_1(q-1) + b_1e$ and $e_2 = a_2(q+1) + b_2e$ with
$a_1,a_2,b_1,b_2 \in \mathbb{Z}$.
Then
$$e_1 \frac{q+1}{2} - e_2\frac{q-1}{2} \equiv (a_1-a_2)\frac{q^2-1}{2}
\pmod{e}.$$
By the first part of the proof, we may assume that $\frac{q^2-1}{e}$ is odd.
Then $q^2-1$ and $e$ have equal $2$-parts. Thus the $2$-parts of
$e_1$ and $q-1$ are equal and smaller than the $2$-part of $e$. Hence $a_1$ is
odd.
Similarly, the $2$-parts of $e_2$ and $q+1$ are equal and smaller than the
$2$-part of $e$.
Thus $a_2$ is also odd. However, then $a_1-a_2$ is even, and our
claim follows. This finishes the proof of our first assertion.

In order to prove the second assertion, suppose that
$m = e_1 < e_2$, and let $k \in \{ 1, \ldots, \frac{q^2-1}{e} \}$ such that
$s_q(ke) = m$.
Then $ke = a + q(e_1-a)$ for some $a \in \{0, \ldots,e_1\}$. Hence
$0  \equiv keq \equiv (e_1-a)+qa \equiv e_1 + (q-1)a \pmod{e}$;
in particular, $0 \equiv e_1 +(q-1)a \equiv e_1 - 2a \pmod{e_2}$, and $-e_1
\leq e_1-2a \leq e_1$.
Thus $e_1=2a$ and $e$ are even. Moreover, we have
$0 \equiv 2a + (q-1)a \equiv (q+1)a \pmod{e}.$
Thus $e \mid (q+1)a = (q+1) \frac{e_1}{2}$, so that $2e \mid (q+1)e_1 \mid
q^2-1$. Hence $\frac{q^2-1}{e}$
is even.
\end{proof}

\begin{example}
Table~\ref{table_meq_middle} gives the values $m(q,e)$
for some larger values of $q$ and $e$.
\end{example}

Next we present some infinite series of examples where $m(q,e)$ can be calculated directly.
The following result is illustrated by
Table~\ref{table_meq_small} and Table~\ref{table_meq_middle}.

\begin{prop}\label{m_for_two_power}
Let $e = 2^k$ for some $k \in \N$ with $k \geq 3$. Moreover, let $q \in \N$ be odd. Then
$$m(q,e) = \left\{ \begin{array}{r@{\,,\quad}l}
                    \gcd(e,q-1) & \hbox{if } q \equiv 1 \pmod{4}, \\
                    2 & \hbox{if } q \equiv -1 \pmod{e}, \\
                    4 & \hbox{otherwise}. 
                   \end{array} \right. $$
\end{prop}

\begin{proof}
It is well-known that
$(\Z/e\Z)^\times = \langle -1+e\Z \rangle \times \langle 5+e\Z \rangle$
and $\ord_e(5) = 2^{k-2}$, $\ord_e(-1) = 2$.
Thus
$$\langle 5+e\Z \rangle = \{b+e\Z \in \Z/e\Z: b \equiv 1 \pmod{4}\}.$$

We also recall the well-known formula 
$$5^{2^t} \equiv 1 + 2^{t+2} \pmod{2^{t+3}}\quad \hbox{for} \quad t \in \N_0.$$
(i)
   Let $q \in \N$ such that $q \equiv 1 \pmod{4}$.
   Then $q \equiv 5^{2^ra} \pmod{e}$ where
   $r \in \{0,1,\ldots,k-2\}$
   and $a \in \N$ is odd,
   so that $\ord_e(q) = 2^{k-2-r}$.
   By \cite[Lemma~\refPartI{same_m_for_same_subgroup}]{BHHK1},
   we have $m:= m(q,e) = m(5^{2^ra},e) = m(5^{2^r},e)$,
   and it is easy to see that $e_1:= \gcd(e,q-1) = \gcd(e,5^{2^ra}-1) = \gcd(e,5^{2^r}-1)$.
   Thus we may assume that $q = 5^{2^r}$.
   Then
   \[
      \langle 5^{2^r} + e\Z \rangle =
         \{b+e\Z \in \Z/e\Z: b \equiv 1 \pmod{2^{r+2}} \},
   \]
   and $e_1 = 2^{r+2}$;
   in particular, $2^{r+2} = e_1 \mid m$
   by \cite[Lemma~\refPartI{prop_m_and_e1}]{BHHK1}.
   It remains to show that $m \leq 2^{r+2}$. 

   By the description of $\langle 5^{2^r}+e\Z \rangle$ above,
   there is $c \in \{0,1,\ldots,2^{k-2-r}-1\}$
   such that $(5^{2^r})^c \equiv 1 - 2^{r+2} \pmod{e}$. 
   Then $e = 2^k \mid q^c + 2^{r+2} -1$,
   and we conclude that $m(q,e) \leq s_q(q^c+2^{r+2}-1) \leq 2^{r+2}$.

(ii) 
   Let $q \in \N$ such that $q \equiv -1 \pmod{e}$. Then $m(q,e) = 2$ by \cite[Lemma~\refPartI{same_m_for_same_subgroup}]{BHHK1}
   and Example~\ref{example_m_eq_2}.

(iii) Let $q \in \N$ such that $q \equiv 3 \pmod{4}$
   and $q \not\equiv -1 \pmod{2^k}$.
   Then $e_1 := \gcd(e,q-1) = 2$.
   Thus $m := m(q,e)$ is even by \cite[Lemma~\refPartI{prop_m_and_e1}]{BHHK1}.
   Since $-1+e\Z \notin \langle q+e\Z \rangle$
   \cite[Lemma~\refPartI{lem_m_eq_2}]{BHHK1} implies that $m>2$,
   i.e. $m \geq 4$, and it remains to show that $m \leq 4$. 

As above, we have $q \equiv -5^{2^ra} \pmod{e}$
where $r \in \{ 0, 1, \ldots, k-3 \}$ and $a \in \N$ is odd.
Moreover, we may assume again that $q \equiv -5^{2^r} \pmod{e}$.
The formula above implies that
\begin{eqnarray*}
   \langle q + e\Z \rangle & = & \langle -5^{2^r}+e\Z \rangle \\
   & = & \{ b + e\Z \in \Z/e\Z: b \equiv -1-2^{r+2} \pmod{2^{r+3}} \}
     \cup \\
   &   & \{ b + e\Z \in \Z/e\Z: b \equiv 1 \pmod{2^{r+3}} \},
\end{eqnarray*}
where the first subset contains those powers of $q + e\Z$
where the exponents are odd
and the second subset contains the powers with even exponents.
Thus there are $c, d \in \{ 0, 1, \ldots, 2^{k-2-r}-1 \}$ such that
\[
   q^c \equiv -1-2^{r+2} \pmod{e} \hbox{ and } q^d \equiv 1 + 2^{r+3} \pmod{e}.
\]
Note that $c$ is odd and $d$ is even,
and that
$2 q^c + q^d + 1 \equiv -2 - 2^{r+3} + 1 + 2^{r+3} + 1 \equiv 0 \pmod{e}$
and $m(q,e) \leq s_q(2q^c+q^d+1) \leq 4$.
\end{proof}

Now we turn to the situation where $e$ is a power of an odd prime.

\begin{prop}\label{prop_m_equals_e1}
Let $e = p^k$ where $p$ is an odd prime and $k \in \N$.
Moreover, let $q \in \N$ such that $q \equiv 1 \pmod{p}$.
Then $m(q,e) = \gcd(e,q-1)$.
\end{prop}

\begin{proof}
The hypothesis $q \equiv 1 \pmod{p}$ implies that $q+e\Z$ is a $p$-element
in $(\Z/e\Z)^\times$.
Since $p$ is odd the Sylow $p$-subgroup of $(\Z/e\Z)^\times$ is generated by
$1+p+e\Z$.
Since $\varphi(e) = p^{k-1}(p-1)$
we have $q \equiv (1+p)^{p^ra} \pmod{e}$ where $r \in \{0,1,\ldots,k-1\}$
and $a \in \N \setminus p\N$.
\cite[Lemma~\refPartI{same_m_for_same_subgroup}]{BHHK1} implies that
$m(q,e) = m((1+p)^{p^r},e)$.
Since
\[
   (1+p)^{p^ra} - 1 =
      \left( (1+p)^{p^r}-1 \right)
      \left( (1+p)^{p^r(a-1)} + \ldots + (1+p)^{p^r} + 1 \right)
\]
where the second factor is not divisible by $p$ we also have
\[
   \gcd(e,q-1) = \gcd(e,(1+p)^{p^r}).
\]
Thus we may assume that $q \equiv (1+p)^{p^r} \pmod{e}$
for some $r \in \{0,1,\ldots,k-1\}$.
We recall that
\[
   (1+p)^{p^{t-1}} \equiv
     1 + p^t \pmod{p^{t+1}} \quad \hbox{for} \quad t \in \N.
\]
Thus
\[
   \langle (1+p)^{p^r} + e \Z \rangle =
    \{b + e \Z \in \Z/e\Z: b \equiv 1 \pmod{p^{r+1}} \}.
\]
Since $\ord_e((1+p)^{p^r}) = p^{k-1-r}$ there is
$c \in \{0,1,\ldots,p^{k-1-r} -1 \}$ such that
\[
   \left( (1+p)^{p^r}\right)^c \equiv
   1 - p^{r+1} \pmod{e}, \quad \hbox{i.e.} \quad e \mid q^c+p^{r+1}-1.
\]
Thus $m(q,e) \leq s_q(q^c + p^{r+1} - 1) \leq p^{r+1} = \gcd(e,q-1)$.
Now \cite[Lemma~\refPartI{prop_m_and_e1}]{BHHK1} implies the result.
\end{proof}

\begin{cor}\label{pierpont_prime_cor}
Let $e = p^k$ where $p$ is an odd prime of the form $p = 1 + 2^a 3^b$
for some nonnegative integers $a$, $b$, and $k \in \N$.
Then
\[
   m(q,e) = \left\{
              \begin{array}{rl}
                \gcd(e, q-1), & \hbox{ if $\ord_e(q)$ is a power of $p$}, \\
                           2, & \hbox{ if $\ord_e(q)$ is even}, \\
                           3, & \hbox{ otherwise}.
              \end{array}
            \right. 
\]
\end{cor}

\begin{proof}
We have $m(q,e) = 2$ if and only if $\ord_e(q)$ is even,
by \cite[Lemma~\refPartI{lem_m_eq_2}]{BHHK1}
and \cite[Remark~\refPartI{rem_m_eq_2_cyclic_case}]{BHHK1} (cf. Table 1),
and Proposition~\ref{prop_m_equals_e1} yields $m(q,e) = \gcd(e, q-1)$
if $q + e \Z$ is a $p$-element in $(\Z/e\Z)^{\times}$.
In the remaining cases, $\ord_e(q)$ is divisible by $3$,
thus \cite[Lemma~\refPartI{lem_m_bounded_by_p}]{BHHK1} yields $m(q,e) \leq 3$,
and equality holds because $m(q,e) \not= 2$.
\end{proof}

\begin{rem}\label{pierpont_prime_remark}
(i)
   Primes of the form $1 + 2^a 3^b$ as above are called
   Pierpont primes, see~\cite{Stewart}.
   It is conjectured that there are infinitely many Pierpont primes.

(ii)
   If $e = p^k$ where $k \in \N$ and $p$ is a Fermat prime
   (e.g. $p \in \{3,5,17\}$) then only the first two cases
   in Corollary~\ref{pierpont_prime_cor} occur.
\end{rem}

\begin{lemma}\label{lemma_hensel_lifting}
 Let $e = p^k$ where $p$ is an odd prime and $k \in \N$. If $\ord_e(q) = \frac{\varphi(e)}{d}$ for a divisor $d$ of $p-1$ then $m(q,e) = m(q,p)$. 
\end{lemma}

\begin{proof}
 Since $p$ is an odd prime and $k \in \N$, $(\Z/e\Z)^\times$ is cyclic of order $\varphi(e) = p^{k-1}(p-1)$. 
 If $q \in \N$ satisfies $\ord_e(q) = \frac{\varphi(e)}{d}$ for a divisor $d$ of $p-1$ then $\langle q + e\Z \rangle = \{x^d + e\Z: x \in \Z \setminus p\Z\}$. 
 Moreover, $\ord_p(q) = \frac{p-1}{d}$ and $\langle q + p\Z = \{x^d + p\Z: x \in \Z \setminus p\Z \}$. 
 Let $m := m(q,p)$. Then there are $x_1,\ldots,x_m \in \Z \setminus p\Z$ such that $x_1^d + \ldots + x_m^d \equiv 0 \pmod{p}$.
 Thus, by Hensel's Lemma (see II.2.2 in \cite{Serre}, for example), there are $y_1,\ldots,y_m \in \Z$ such that $y_1^d + \ldots + y_m^d \equiv 0 \pmod{e}$ and $y_i \equiv x_i \pmod{p}$
 for $i=1,\ldots,m$; in particular, $y_1,\ldots,y_m \notin p\Z$. This shows that $m(q,e) \leq m = m(q,p)$. 
 Since $m(q,p) \leq m(q,e)$ by \cite[Lemma~\refPartI{same_m_for_same_subgroup}]{BHHK1}, the result follows.
\end{proof}
 
\begin{cor}\label{cor_squares_mod_prime_powers}
 Let $e = p^k$ where $p$ is an odd prime and $k \in \N$. If $\ord_e(q) = \frac{\varphi(e)}{2}$ then
 $$ m(q,e) = \left\{
              \begin{array}{rl}
                           2, & \hbox{ if $p \equiv 1 \pmod{4}$}, \\
                           3, & \hbox{ if $p \equiv 3 \pmod{4}$.} \\
              \end{array}
            \right. $$
\end{cor}

\begin{proof}
 By the lemma above and its proof, we have $m(q,e) = m(q,p)$ and $\ord_p(q) = \frac{p-1}{2}$. Thus we may assume that $k=1$, i.e. $e=p$. 
 Since $(\Z/p\Z)^\times$ is cyclic of order $\varphi(p) = p-1$, $\langle q + p\Z\rangle$ consists of the squares in $(\Z/p\Z)^\times$. 
 It is well-known that there are $x_1,x_2 \in \Z$ such that $x_1^2 + x_2^2 \equiv -1 \pmod{p}$ (cf. IV.1.7 in 
 \cite{Serre}, for example), i.e. $x_1^2 + x_2^2 + 1^2 \equiv 0 \pmod{p}$. Hence $m(q,p) \leq 3$. Moreover, by \cite[Remark~\refPartI{rem_m_eq_2_cyclic_case}]{BHHK1}, we have
 $m(q,p) = 2$ if and only if $\ord_p(q)$ is even, i.e. if and only if $p \equiv 1 \pmod{4}$.
\end{proof}

\begin{rem}
 Let $e$ be a prime, and let $q \in \N$ such that $\ord_e(q) = \frac{e-1}{k}$ where $k \in \N$ divides $e-1$ and $e > (k-1)^4$. 
 Then, by the main theorem in \cite{Small}, every element in $\Z/e\Z$ (in particular, $-1+e\Z$) is a sum of two $k$-th powers. 
 Since $\langle q+e\Z \rangle$ is the set of all $k$-th powers in $(\Z/e\Z)^\times$, this implies that $0+e\Z$ is a sum of at most three elements in $\langle q + e\Z \rangle$. 
 Thus $m(q,e) \leq 3$.
 
 For example, if $k = 4$ then $m(q,e) \leq 3$ for every prime $e > 81$. 
 Explicit computations show that, for $k = 4$, the only cases where $m(q,e) > 3$ are as follows:
 \begin{itemize}
  \item $e=5$, $q \equiv 1 \pmod{5}$, $m(q,e)=5$;
  \item $e = 29$, $\langle q + 29\Z \rangle = \langle 7 + 29\Z \rangle$, $m(q,e) = 4$.
 \end{itemize}

 \end{rem}

The smallest (odd) prime that is not a Pierpont prime is $p=11$. Here we get the following.

\begin{prop}\label{prop_e_11}
 Let $e = 11^k$ for some $k \in \N$. Then
 \[
   m(q,e) = \left\{
              \begin{array}{rl}
                \gcd(e, q-1), & \hbox{ if $\ord_e(q)$ is a power of $11$}, \\
                           2, & \hbox{ if $\ord_e(q)$ is even}, \\
                           3, & \hbox{ if $\ord_e(q) = 5 \cdot 11^{k-1}$,} \\
                           5, & \hbox{ otherwise}.
              \end{array}
            \right. 
   \]
\end{prop}

\begin{proof}
 Note first that $(\Z/e\Z)^\times$ is cyclic of order $2 \cdot 5 \cdot 11^{k-1}$ and generated by $2+e\Z$. 
 
 If $\ord_e(q)$ is a power of $11$ then $m(q,e) = \gcd(e,q-1)$ by Proposition~\ref{prop_m_equals_e1}, and if $\ord_e(q)$ is even then $m(q,e) = 2$ by \cite[Lemma~\refPartI{lem_m_eq_2}]{BHHK1}.
 
 In all other cases, we have $\ord_e(q) = 5 \cdot 11^{l}$ for some $l<k$. Then $\gcd(11,q^{11^l}-1) = \gcd(11,q-1) = 1$, and therefore $\gcd(e,q^{11^l}-1) = 1$. Thus $m(q,e) \leq 5$ by
 \cite[Lemma~\refPartI{lem_m_bounded_by_p}]{BHHK1}, and $m(q,e) > 2$ by \cite[Lemma~\refPartI{lem_m_eq_2}]{BHHK1}. 
 
 If $\ord_e(q) = 5 \cdot 11^{k-1} = \varphi(e)/2$ then the result follows from Lemma~\ref{cor_squares_mod_prime_powers}. 
 
 Thus we may assume that $k \geq 2$. Then \cite[Lemma~\refPartI{same_m_for_same_subgroup}]{BHHK1} and explicit computations show: $m(4^{11},e) \geq m(4^{11},11^2) = 5$, i.e. $m(4^{11},e) = 5$. 
 Hence $m(4^{11^n},e) \geq m(4^{11},e) = 5$ for $n \in \N$, again by \cite[Lemma~\refPartI{same_m_for_same_subgroup}]{BHHK1}, and the result follows. 
\end{proof}

Our next result is similar to Proposition~\ref{prop_m_equals_e1}.

\begin{prop}\label{m_eq_e1_again}
 Let $e = 2p^k$ where $p$ is an odd prime and $k \in \N$.
 Moreover, let $1<q \in \N$ such that $\ord_e(q)$ is a power of $p$.      
 Then $m(q,e) = \gcd(e,q-1)$.
\end{prop}

\begin{proof}
Omitted.
\end{proof}

\begin{rem}
Suppose that $e = 2p^k$ where $k \in \N$ and $p$ is an odd Pierpont prime, and write $p = 1 + 2^a3^b$ where $a,b \in \N_0$. Then $\varphi(e) = p^{k-1}2^a3^b$.
Let $q \in \N$ such that $\gcd(q,e) = 1$. If ${\ord}_e(q)$ is even then $m(q,e) = 2$ by \cite[Remark~\refPartI{rem_m_eq_2_cyclic_case}]{BHHK1}. If 
$\ord_e(q)$ is a power of $p$ then $m(q,e) = \gcd(e,q-1)$ by Proposition~\ref{m_eq_e1_again}. It remains to deal with the case where $\ord_e(q)$ is odd and 
divisible by $3$. Then $e_1 := \gcd(e,q-1) = 2$, and $m(q,e) \leq e_1m(q,\frac{e}{e_1}) = 2m(q,p^k) \leq 6$ by Corollary~\ref{pierpont_prime_cor}. Since 
$2 = e_1 \mid m(q,e)$ by \cite[Lemma~\refPartI{prop_m_and_e1}]{BHHK1} we conclude that $m(q,e) \in \{4,6\}$ in this case. 
Moreover, it is easy to check that $m(q,e) = 6$ whenever $\ord_e(q) = 3$ and $e > 14$. 
\end{rem}

\begin{example}
Table~\ref{table_meq_large} displays some more values $m(q,e)$.
Here we list only one generator $q+e\mathbb{Z}$ for
each cyclic subgroup of $(\mathbb{Z}/e\mathbb{Z})^\times$.
\end{example}

The following result may also be of interest;
it is related to \cite[Lemma~\refPartI{prop_m_and_e1}~(iv)]{BHHK1}.
Here we denote by $\Phi_n \in \Q[X]$ the $n$-th cyclotomic polynomial.

\begin{prop}\label{prop_e_divides_Phi_n_and_m_smaller_than_n}
Let $n$ be a prime number. Then there are only finitely many $e \in \N$ such that $e \mid \Phi_n(q)$ (in particular, $\ord_e(q) \mid n$)
and  $m(q,e) < n$ for some $q \in \N$. 
\end{prop}

\begin{proof}
We fix a prime number $n$ and an integer $m \in \{1,\ldots,n-1\}$.
Suppose that $q,e \in \N$ satisfy $e \mid \Phi_n(q)$ and $m(q,e) = m$.
Then there are $i_1,\ldots,i_m \in \N_0$ such that 
$$(\ast) \quad\quad\quad q^{i_1} + \ldots + q^{i_m} \equiv 0 \pmod{e}.$$
Since $q^n \equiv 1 \pmod{e}$ we may assume that $i_1,\ldots,i_m \in \{0,\ldots,n-1\}$. 
Since $m<n$ there exists $j \in \{0,\ldots,n-1\} \setminus \{i_1,\ldots,i_m\}$. 
Since we can multiply $(\ast)$ by $q^{n-1-j}$ we may assume that $j = n-1$, i.e.\ $i_1,\ldots,i_m \in \{0,\ldots,n-2\}$. 
(Of course, we may then also assume that $0 = i_1 \leq \ldots \leq i_m \leq n-2$.)
Thus there are only finitely many possibilities for the $m$-tuple $(i_1,\ldots,i_m)$. 
Consider the polynomial $g(X) := X^{i_1} + \ldots + X^{i_m} \in \Q[X]$. 
Since $\Phi_n(X)$ is irreducible of degree $n-1$ there are $a(X), b(X) \in \Q[X]$ such that $a(X)\Phi_n(X) + b(X)g(X) = 1$.
We fix $d \in \N$ such that $da(X), db(X) \in \Z[X]$.
Then $e \mid da(q)\Phi_n(q) + db(q)g(q) = d$.
Hence there are only finitely many possibilities for $e$, as claimed.
\end{proof}

\begin{example}\label{example_Phi5}
(i)
Let $q, e \in \N$ such that $e \mid \Phi_5(q)$. 
Carrying out the calculations in the proof of
Proposition~\ref{prop_e_divides_Phi_n_and_m_smaller_than_n}
for $1 \leq m \leq 4$, we obtain the following values $(d, m)$,
where we choose the minimal possible $d$:
$(1, m)$ and $(m, m)$ for all $m$, $(2, 4)$,
$(11, 3)$, $(11, 4)$, and $(61, 4)$.
Since divisors of $\Phi_d(q)$ are odd for odd $d$
(and are not divisible by $3$ if additionally $d$ is not divisible by $3$),
we get $e = 1$ if $m(q, e) = 1$,
$e = 11$ if $m(q, e) = 3$, and
$e \in \{ 11, 61 \}$ if $m(q, e) = 4$.
(And Proposition~\ref{prop_e_11} or Table~\ref{table_meq_small} shows
that we cannot have $m(q,e) = 4$
in case $e = 11$.)
We will need this result later on.

(ii)
Let $q, e \in \N$ such that $1 < e \mid \Phi_7(q)$, and let $m = m(q,e)$.
We proceed as in (i), and discard $2$-parts and $3$-parts of the values
for $d$.
Moreover, we can discard those candidates $d$ with the property that
$7$ does not divide $\varphi( d )$,
since the relevant divisors $e$ of $d$ must satisfy $\ord_e(q) = 7$
for some prime residue $q$ modulo $e$;
this criterion excludes $(d, m) \in \{ (13, 5), (41, 6) \}$.
We are left with the following list.
\[
   \begin{array}{ll}
     m = 3: & e \in \{ 43 \} \\
     m = 4: & e \in \{ 29, 71, 547 \} \\
     m = 5: & e \in \{ 29, 43, 113, 197, 421, 463, 3277 \} \\
     m = 6: & e \in \{ 29, 43, 71, 113, 197, 211, 379, 449, 463,
                       757, 2689, 3053, 13021 \}
   \end{array}
\]
(Again, computing $m(q,e)$ shows that
$e = 29$ can occur only for $m = 4$ (see Table~\ref{table_meq_small}),
$e = 43$ can occur only for $m = 3$ (see Table~\ref{table_meq_middle}),
$e = 71$ can occur only for $m = 4$ (see Table~\ref{table_meq_large}), and
$e \in \{ 113, 197, 463 \}$ can occur only for $m = 5$.)
\end{example}

Let again $q,e \in \N$ such that $\gcd(q,e) = 1$.
Moreover, let $n \in \N$ such that $e \mid q^n-1$ (e.g. $n = \ord_e(q)$),
and set $z := \frac{q^n-1}{e}$.
In order to avoid trivialities, we also suppose that $q>1$ and $1 < e < q^n-1$.
Then \cite[Proposition~\refPartI{m_and_digit_sums}]{BHHK1} implies that
$$ m(q,e) = \min \{s_q(ke): k=1,\ldots,z\}
          = \min \{s_q(ke): k=1,\ldots,z-1\}.$$
Our next aim is to derive another description of $m(q,e)$.
For this we introduce some more notation.
For $x \in \Z$, we define $\overline{x} \in \Z$ by
$\overline{x} \equiv x \pmod{z}$ and $0 \leq \overline{x} < z$.

\begin{prop}\label{prop_coeffs_z}
Let $k \in \N$ such that $k < z$.
Then $ke$ has the $q$-adic expansion
\[
   ke = \sum_{i=1}^n
     \frac{\overline{k q^{n-i}} q - \overline{k q^{n-i+1}}}{z} q^{i-1}
      = \sum_{i=1}^n
     \left\lfloor \frac{\overline{k q^{n-i}} q}{z} \right\rfloor q^{i-1}.
\]
Thus
$$s_q(ke) = \frac{q-1}{z} \sum_{i=1}^n \overline{kq^i}
= \frac{q-1}{z} \frac{n}{\mathrm{ord}_z(q)} \sum_{i=1}^{\mathrm{ord}_z(q)} \overline{kq^i},$$
and
$$m(q,e) = \frac{q-1}{z} \frac{n}{\mathrm{ord}_z(q)}\min \left\{ \sum_{i=1}^{\mathrm{ord}_z(q)} \overline{kq^i}: 1 \leq k < z \right\}.$$
\end{prop}

\begin{proof}
Set $c_{k,i} = \overline{k q^{n-i}}$, for $0 \leq i \leq n$,
and denote the coefficient of $q^{i-1}$ in the above summation by $a_{k,i}$.
Then $z a_{k,i} = c_{k,i} q - c_{k,i-1}$ for $i=1,\ldots,n$,
$c_{k,0} = c_{k,n} = k$, and
\[
   z \sum_{i=1}^n a_{k,i} q^{i-1} =
       \sum_{i=1}^n c_{k,i} q^i - \sum_{i=1}^n c_{k,i-1} q^{i-1} =
       c_{k,n} q^n - c_{k,0} = k \left( q^n - 1 \right) = z k e
\]
holds, as claimed.

The $a_{k,i}$ are integers because
$\overline{k q^{n-i}} q - \overline{k q^{n-i+1}} \equiv 0 \pmod{z}$.

The $a_{k,i}$ are nonnegative because
$\overline{k q^{n-i}} q \geq
 \overline{\overline{k q^{n-i}} {q}} =
 \overline{k q^{n-i+1}}$.

We have $a_{k,i} < q$ because
$\overline{k q^{n-i}} q - \overline{k q^{n-i+1}} < z q$.
Thus the $q$-adic expansion of $ke$ has the desired form. Hence
\[
   s_q(ke) = \sum_{i=1}^n a_{k,i} = \frac{q-1}{z} \sum_{i=1}^n c_{k,i}
           = \frac{q-1}{z} \sum_{i=1}^n \overline{k q^i},
\]
and $m(q,e)$ is the minimum of these values,
for the admissible values of $k$.
\end{proof}

\begin{rem}\label{rem_calculating_expansions}
(i) A natural way to derive the statement of Proposition~\ref{prop_coeffs_z}
is as follows.
Dividing the obvious $q$-adic expansion of $q^n-1$ by $z$, we get
\[
   \frac{q^n-1}{z} = \frac{q-1}{z} q^{n-1} + \frac{q-1}{z} q^{n-2} + \cdots +
                     \frac{q-1}{z} q^0.
\]
If $q-1$ is not divisible by $z$ then these coefficients aren't integers,
and we adjust them iteratively:
Replacing $(q-1)/z$ by $(q - \overline{q})/z$ in the coefficient
of $q^{n-1}$ yields an integer, and can be compensated in the summation
by choosing $(\overline{q} q - 1)/z$ as the coefficient of $q^{n-2}$.
Next we replace this coefficient by
$(\overline{q} q - \overline{q^2})/z$ and adjust the coefficient of
$q^{n-3}$ accordingly.
Repeating this process,
we get $(\overline{q^{n-1}} q - 1)/z$ as the coefficient of $q^0$,
which is already an integer because $\ord_z(q)$ divides $n$.

(ii) Note that the cyclic subgroup $H := \langle q+z\mathbb{Z} \rangle$ of the multiplicative group
$G := (\mathbb{Z}/z\mathbb{Z})^\times$
acts on the additive group $\mathbb{Z}/z\mathbb{Z}$ by multiplication, and that
$$\sum_{i=1}^{\mathrm{ord}_z(q)} \overline{kq^i} = |H_k| \sum_{x \in B} \overline{x}$$
where $B$ is the $H$-orbit of $k+z\mathbb{Z}$ and
$$H_k := \{h+z\mathbb{Z} \in H: kh \equiv k \pmod{z} \}$$
is the stabilizer of $k+z\mathbb{Z}$ in $H$.

(iii) Now suppose, in addition, that $-1 + z\mathbb{Z} \in H = \langle q+z\mathbb{Z} \rangle$. Then $B = -B$ and $\overline{-x} =
z - \overline{x}$ for $x \in B$. Thus
$$2 \sum_{x \in B} \overline{x} = \sum_{x \in B} \overline{x} + \sum_{x \in B} \overline{-x} =
\sum_{x \in B} \overline{x} + \sum_{x \in B} z - \overline{x} = |B| z,$$
and $\sum_{x \in B} \overline{x} = |H:H_k| \frac{z}{2}$. Hence Proposition~\ref{prop_coeffs_z} implies that
$$s_q(ke) = n \frac{q-1}{2}.$$
Since this expression is independent of $k$ we conclude that
$$m(q,e) = n \frac{q-1}{2}$$
in this case.

(iv) The second expression for the $q$-adic expansion of $k e$
that is stated in Proposition~\ref{prop_coeffs_z}
can be used to compute the coefficients for $i = n, n-1, \ldots, 1$,
without computing the number $k e$.
Note that in typical examples (see Remark~\ref{ref_database_examples}),
$q$ and $z$ are small numbers, whereas $e$ can be quite large.
\end{rem}

\begin{example}\label{example_s_q_equal}
(i) If $q \equiv 1 \pmod{z} $ then $ke = \sum_{i=1}^n \frac{k(q-1)}{z} q^{i-1}$ for $k = 1, \ldots, z-1$, and $s_q(ke) = \frac{kn(q-1)}{z}$ for these $k$. 
Hence $m(q,e) = ne\frac{q-1}{q^n-1}$, as in Corollary~\ref{cor_m_functional_equation}.

(ii) If $q \equiv -1 \pmod{z}$, if $n=2$ and $z > 2$ then
$$ke = \frac{(z-k)q-k}{z} + \frac{kq - (z-k)}{z} q,$$
so that $s_q(ke) = q-1$ for $k=1,\ldots,z-1$, and $m(q,e) = q-1$.

(iii) Suppose that $(\Z/z\Z)^\times$ is cyclic and generated by $q+z\Z$.
Then $s_q(ke) = n \frac{q-1}{2}$
for $k=1,\ldots,z-1$, and $m(q,e) = n \frac{q-1}{2}$.

(iv) Suppose that $z = p^a$ where $p$ is a prime with $p \equiv 1 \pmod{4}$,
 and $a \in \N$.
Moreover, suppose that $\ord_z(q) = \frac{\varphi(z)}{2} = p^{a-1} \frac{p-1}{2}$.
Then $(\Z/z\Z)^\times$ is cyclic, and $-1+z\Z \in \langle q+z\Z \rangle$.
Thus Remark~\ref{rem_calculating_expansions} implies that
$$m(q,e) = s_q(ke) = n \frac{q-1}{2} \quad \hbox{for} \quad k=1,\ldots,z-1.$$
\end{example}

\begin{prop}\label{prop_large_order_of_q_modulo_prime}
Suppose that $z$ is a prime with $z \equiv -1 \pmod{4}$,
and that $\ord_z(q) = \frac{z-1}{2}$.
Then $|\{ s_q(ke); 1 \leq k \leq z-1 \}| = 2$.
\end{prop}

\begin{proof}
Let $k \in \{1,\ldots,z-1\}$. Then Proposition~\ref{prop_coeffs_z} implies:
$$s_q(ke) = \frac{q-1}{z} \frac{2n}{z-1} \sum_{i=1}^{\ord_z(q)} \overline{kq^i}.$$
For $\epsilon \in \{ \pm 1 \}$,
we set $G_\epsilon := \{x+z\Z \in (\Z/z\Z)^\times: \left( \frac{x}{z} \right) =
\epsilon \}$ where $\left( \frac{x}{z} \right)$ denotes the Legendre symbol.
Then
$$\{kq^i + z\Z : i = 1,\ldots,\ord_z(q)\} \in \{G_+, G_-\}.$$
Moreover, by a result of Dirichlet (cf. \cite[Chap.~6, equ.~(19)]{Davenport}),
we have
$$\sum_{x+z\Z \in G_+} \overline{x} < \sum_{x + z\Z \in G_-} \overline{x}.$$
The result follows.
\end{proof}


\section{Lower bounds}

Let $q,n,e \in \N$ such that $q>1$ and $e \mid q^n-1$. We set $z := \frac{q^n-1}{e}$. Moreover, let $F$ be a field, and let $A = A(q,n,e)$ be the 
$F$-algebra of \cite{BHHK1}, Section 3. We denote by $J := {\mathrm J}(A)$ the Jacobson radical of $A$.
In this section we establish certain lower bounds for $\LL(A)$. In certain cases, these lower bounds coincide with
the upper bound established in \cite[Theorem~\refPartI{thm_upper_and_lower_bounds_I}]{BHHK1}.

\begin{prop}\label{prop_n1_plus_n2}
 Let $n_1,n_2 \in \N$ such that $e \mid q^{n_1}-1$ and $e \mid q^{n_2}-1$. Then
 $$\LL(A(q,n_1+n_2,e)) \geq \LL(A(q,n_1,e)) + \LL(A(q,n_2,e)) -1.$$
 Thus $\LL(A(q,rn,e)) \geq r \cdot \LL(A(q,n,e)) - r + 1$, for $r \in \N$.
\end{prop}

\begin{proof}
Note first that $q^{n_1 + n_2} = q^{n_1}q^{n_2} \equiv 1 \cdot 1 \equiv 1 \pmod{e}$. Then observe that the $F$-algebra 
$F[x_1,\ldots,x_{n_1},x_{n_1+1}, \ldots, x_{n_1+n_2}]$ contains the subalgebras $F[x_1,\ldots,x_{n_1}]$ and $F[x_{n_1+1},\ldots,x_{n_1+n_2}]$. We consider
$A(q,n_1,e)$ as a subalgebra of $F[x_1,\ldots,x_{n_1}]$, as usual, and $A(q,n_2,e)$ as a subalgebra of $F[x_{n_1+1},\ldots,x_{n_1+n_2}]$, via a shift of
indices. Then $A(q,n_1,e)$ and $A(q,n_2,e)$ become subalgebras of $A(q,n_1+n_2,e)$; this is obvious for $A(q,n_1,e)$, and if a monomial $x_{n_1+1}^{i_1}
\ldots x_{n_1+n_2}^{i_{n_2}}$ satisfies 
$$i_1 + qi_2 + \ldots + q^{n_2-1}i_{n_2} \equiv 0 \pmod{e}$$ 
then also $q^{n_1}i_1 + q^{n_1+1}i_2 + \ldots + q^{n_1+n_2-1}i_{n_2} \equiv 0 \pmod{e}$.

For $j=1,2$, let $t_j := \LL(A(q,n_j,e))-1$, and let $y_j$ be a nonzero product of $t_j$ basis elements in $\J(A(q,n_j,e))$. Then $y_1y_2$
is a nonzero product of $t_1t_2$ basis elements in $\J(A(q,n_1+n_2,e))$. Thus $\LL(A(q,n_1+n_2,e)) > t_1+t_2$
which implies the first inequality in Proposition~\ref{prop_n1_plus_n2}.
The second inequality follows by induction on $r$.
\end{proof}

We obtain the following consequence.

\begin{cor}\label{cor_n1_plus_n2}
 For $i=1,2$, let $n_i \in \N$ such that $e \mid q^{n_i}-1$. Moreover, suppose that $\LL(A(q,n_i,e)) = \lfloor n_i \frac{q-1}{m} \rfloor +1$
 for $i=1,2$ where $m = m(q,e)$. If $m \mid n_1(q-1)$ then $\LL(A(q,n_1+n_2,e)) =  \lfloor (n_1+n_2) \frac{q-1}{m} \rfloor + 1$.
\end{cor}

\begin{proof}
By Proposition~\ref{prop_n1_plus_n2},
the hypotheses of Corollary~\ref{cor_n1_plus_n2} imply:
\begin{eqnarray*}
 \LL(A(q,n_1+n_2,e)) &\geq &\LL(A(q,n_1,e)) + \LL(A(q,n_2,e)) - 1 \cr
                           &= & \left\lfloor n_1 \frac{q-1}{m} \right\rfloor + \left\lfloor n_2 \frac{q-1}{m} \right\rfloor + 1 \cr
                           &= &\left\lfloor n_1 \frac{q-1}{m} + n_2 \frac{q-1}{m} \right\rfloor +1 = \left\lfloor (n_1+n_2) \frac{q-1}{m} \right\rfloor +1.
\end{eqnarray*}
Thus the result follows from
\cite[Theorem~\refPartI{thm_upper_and_lower_bounds_I}~(i)]{BHHK1}.
\end{proof}

\begin{rem}\label{rem_ll_for_rn_from_ll_for_n}
Suppose that $m := m(q,e) \mid n(q-1)$ and $e$ divides $q^n-1$.
Then Corollary~\ref{cor_n1_plus_n2} implies, by induction:
 If $\LL(A(q,n,e)) = \lfloor n \frac{q-1}{m} \rfloor + 1$ then $\LL(A(q,rn,e)) = \lfloor rn \frac{q-1}{m} \rfloor +1$, for $r \in \N$.
\end{rem}

These results lead to the following reduction:

\begin{prop}\label{prop_reduction_ll_for_n}
 Let $q,e \in \N$ such that $q>1$ and $\gcd(q,e) = 1$. Moreover, let $m := m(q,e)$, and let $N \in \N$ such that $\ord_e(q) \mid N$ and
 $m \mid N(q-1)$. If $\LL(A(q,n,e)) = \lfloor  n \frac{q-1}{m} \rfloor + 1$ for all $n \in \N$ with $\ord_e(q) \mid n \leq N$ then
 $\LL(A(q,n,e)) = \lfloor n \frac{q-1}{m} \rfloor +1$ for all $n \in \N$ with $e \mid q^n-1$. 
\end{prop}

\begin{proof}
Suppose that $\LL(A(q,n,e)) = \lfloor n \frac{q-1}{m} \rfloor + 1$ for all $n \in \N$ with $\ord_e(q) \mid n \leq N$.
Moreover, let $n \in \N$ with $e \mid q^n-1$ and $n > N$. Then there are $a,r \in \N$ such that $n = aN+r$ and $1 \leq r \leq N$. Thus 
$1 \equiv q^n \equiv (q^N)^aq^r \equiv q^r \pmod{e}$.
By our assumption and Remark~\ref{rem_ll_for_rn_from_ll_for_n},
this implies that $\LL(A(q,r,e)) = \lfloor r \frac{q-1}{m}
\rfloor + 1$ and $\LL(A(q,aN,e)) = \lfloor aN \frac{q-1}{m} \rfloor + 1$. Hence $\LL(A(q,n,e)) = \lfloor (aN+r) \frac{q-1}{m} \rfloor + 1
= \lfloor n \frac{q-1}{m}\rfloor + 1$,
by Corollary~\ref{cor_n1_plus_n2}.
\end{proof}

It is easy to see that one may take $N := \frac{m}{d} \ord_e(q)$ where $d := \gcd(m,(q-1)\ord_e(q))$, so that $N \mid \frac{m}{e_1} \ord_e(q)$ where $e_1 := \gcd(e,q-1)$.

\begin{prop}\label{prop_factorization_into_k_monomials}
Let $q, Q, e \in\N$ such that $\gcd(q,e) = 1 = \gcd(Q,e)$, and suppose that $q + e \Z$ and $Q + e \Z$ generate the same subgroup
of $(\Z/e\Z)^{\times}$.
Let $n$ be a multiple of $\ord_e(q)$ and $1 \leq a < \min\{ q, Q \}$.
Then $0 \not= (x_1 x_2 \cdots x_n)^a \in \J(A(q,n,e))$
is a product of $k$ monomials in $\J(A(q,n,e))$
if and only if 
$0 \not= (x_1 x_2 \cdots x_n)^a \in \J(A(Q,n,e))$
is a product of $k$ monomials in $\J(A(Q,n,e))$.
\end{prop}

\begin{proof}
There is a permutation $\pi$ of $\{ 0, 1, \ldots, n-1 \}$
such that $q^t \equiv Q^{\pi(t)} \pmod{e}$ for $t = 0, 1, \ldots, n-1$.
For any choice of exponents $a_0, a_1, \ldots, a_{n-1}$
in $\{ 0, 1, \ldots, a \}$,
we have $x_1^{a_0} x_2^{a_1} \cdots x_n^{a_{n-1}} \in \J(A(Q,n,e))$
if and only if 
$0 \equiv \sum_{i=0}^{n-1} a_i Q^i \equiv \sum_{i=0}^{n-1} a_{\pi(i)} q^i \pmod{e}$,
which happens if and only if
$x_1^{a_{\pi(0)}} x_2^{a_{\pi(1)}} \cdots x_n^{a_{\pi(n-1)}} \in \J(A(q,n,e))$.
Given a factorization of $(x_1 x_2 \cdots x_n)^a$ into $k$ factors
in $\J(A(Q,n,e))$,
permuting the exponents of each factor with $\pi$ yields a factorization
into $k$ factors in $\J(A(q,n,e))$.
The same argument works in the other direction.
\end{proof}

\begin{rem}\label{rem_Q_congruent_q_mod_m}
Suppose that we are in the situation of
Proposition~\ref{prop_factorization_into_k_monomials},
that is, $q + e \Z$ and $Q + e \Z$ generate the same subgroup
of $(\Z/e\Z)^{\times}$.
If $Q = q + m a$, with $m = m(q,e) = m(Q,e)$ and $a \in \N$,
then we have
\[
   \left\lfloor \frac{n (Q-1)}{m} \right\rfloor =
   \left\lfloor \frac{n (q-1)}{m} \right\rfloor + n a.
\]
Setting $t = \LL( A(q,n,e) ) - 1$,
we have $0 \neq (x_1 x_2 \cdots x_n)^{q-1} \in \J( A(q,n,e) )^t$.
By Proposition~\ref{prop_factorization_into_k_monomials}, this implies
$$ 0 \neq (y_1 \ldots y_n)^{Q-1} = (y_1 \ldots y_n)^{q-1}(y_1\ldots y_n)^{ma} \in \J(A(Q,n,e))^t \J(A(Q,n,e))^{na}$$
where we consider $A(Q,n,e)$ as a subalgebra of $F[y_1,\ldots,y_n] = F[Y_1, \ldots, Y_n]/(Y_1^q,\ldots,Y_n^q)$.
Thus $\LL( A(Q,n,e) ) \geq t + 1 + n a = \LL( A(q,n,e) ) + n a$.
This implies:
If the upper bound from
\cite[Theorem~\refPartI{thm_upper_and_lower_bounds_I}]{BHHK1}
is attained for $A(q,n,e)$ then it is also attained for $A(Q,n,e)$.
\end{rem}

\begin{prop}\label{prop_ll_for_shifted_q}
Let $q,n,e,m$ be as usual, and let $l := \lcm(e,m)$. Then 
$$\LL(A(q+l,n,e)) \geq \frac{nl}{m} + \LL(A(q,n,e)).$$
Thus, if $\LL(A(q,n,e)) = \lfloor n \frac{q-1}{m} \rfloor + 1$ then $\LL(A(q+l,n,e)) = \lfloor n \frac{q+l-1}{m} \rfloor + 1$.
\end{prop}

\begin{proof}
We write the nonzero monomial $x_1^{q+l-1} \ldots x_n^{q+l-1} \in A(q+l,n,e)$ as a product of the two monomials $x_1^{q-1} \ldots x_n^{q-1}$ and
$x_1^l \ldots x_n^l$ in $A(q+l,n,e)$. Let $x_1^{i_1} \ldots x_n^{i_n}$ be a monomial of degree $m := m(q,e) = m(q+l,e)$. The product of the $n$ cyclic shifts of
$x_1^{i_1} \ldots x_n^{i_n}$ is $x_1^m \ldots x_n^m$, so that $x_1^l \ldots x_n^l = (x_1^m \ldots x_n^m)^{l/m}$ is a product of $\frac{nl}{m}$ monomials in 
$\J(A(q+l,n,e))$. 

In order to distinguish between $A(q,n,e)$ and $A(q+l,n,e)$ we consider $A(q,n,e)$ as a subalgebra of $F[y_1,\ldots,y_n] = F[Y_1, \ldots, Y_n]/(Y_1^q,\ldots,Y_n^q)$.
The nonzero monomial $y_1^{q-1} \ldots y_n^{q-1} \in A(q,n,e)$ can be written as a product of $t := \LL(A(q,n,e)) -1$ monomials in $\J(A(q,n,e))$.
Thus the monomial $x_1^{q-1} \ldots x_n^{q-1} \in A(q+l,n,e)$ can be written as a product of the corresponding $t$ monomials in $\J(A(q+l,n,e))$; note that
$$j_1 + (q+l)j_2 + \ldots + (q+l)^{n-1}j_n \equiv j_1+ qj_2 + \ldots + q^{n-1}j_n \pmod{e}$$
for $j_1,\ldots,j_n \in \Z$. Thus $x_1^{q+l-1} \ldots x_n^{q+l-1}$ can be written as a product of $\frac{nl}{m} + t$ monomials in $\J(A(q+l,n,e))$. 
Hence
$$\LL(A(q+l,n,e)) \geq \frac{nl}{m} + t + 1 = \frac{nl}{m} + \LL(A(q,n,e)).$$
If $\LL(A(q,n,e)) = \lfloor n \frac{q-1}{m} \rfloor + 1$ then this implies 
$$\LL(A(q+l,n,e)) \geq \frac{nl}{m} + \left\lfloor n \frac{q-1}{m} \right\rfloor + 1 = \left\lfloor \frac{nl}{m} + n \frac{q-1}{m} \right\rfloor + 1 = \left\lfloor n \frac{q+l-1}{m} \right\rfloor + 1,$$
and the result follows from
\cite[Theorem~\refPartI{thm_upper_and_lower_bounds_I}]{BHHK1}.
\end{proof}

\begin{rem}\label{rem_reduction_for_fixed_e_to_finitely_many_q}
Let $n,e,Q \in \N$ such that $Q \not\equiv 1 \equiv Q^n \pmod{e}$, and let $m := m(Q,e)$.
Then, by Proposition~\ref{prop_ll_for_shifted_q},
we have $\LL(A(q,n,e)) =
\lfloor n \frac{q-1}{m} \rfloor + 1$ for all $q \in \N$ with $q \equiv Q \pmod{e}$ if and only if $\LL(A(q,n,e)) = \lfloor n \frac{q-1}{m} \rfloor + 1$
for all $q \in \N$ with $q \equiv Q \pmod{e}$ and $q \leq \lcm(e,m)$.
\end{rem}

\begin{prop}
Let $q,n,e$ be as usual, and let $f \mid e$.
Then $A(q,n,e)$ can be viewed as a subalgebra of $A(q,n,f)$;
in particular, $\LL(A(q,n,e)) \leq \LL(A(q,n,f))$. 
\end{prop}

\begin{proof}
 By definition, $A(q,n,e)$ is the $F$-subalgebra of $F[x_1,\ldots,x_n]$ generated by all monomials $x_1^{i_1} \ldots x_n^{i_n}$ such that $i_1 + q i_2 + \ldots + q^{n-1}i_n \equiv 0 \pmod{e}$.
 (Recall that $x_1^q = \ldots = x_n^q = 0$.) Similarly, $A(q,n,f)$ is the $F$-subalgebra of $F[x_1,\ldots,x_n]$ generated by all monomials $x_1^{i_1} \ldots x_n^{i_n}$ such that $i_1 + qi_2 + \ldots +
 q^{n-1}i_n \equiv 0 \pmod{f}$. Thus $A(q,n,e) \subseteq A(q,n,f)$ and $\J(A(q,n,e)) \subseteq \J(A(q,n,f))$; in particular, $\LL(A(q,n,e)) \leq \LL(A(q,n,f))$.
\end{proof}

\begin{prop}
Let $q,n,e$ be as usual, and let $k \in \N$. Then
\[
   \LL(A(q,kn,e)) \leq \LL(A(q^k,n,e)).
\]
\end{prop}

\begin{proof}
 Note first that both $A(q,kn,e)$ and $A(q^k,n,e)$ have dimension $z+1$ where
 $$z := (q^{kn}-1)/e.$$ 
 We set $Q := q^k$ and denote the standard bases of $A(q,kn,e)$ and $A(Q,n,e)$ by $b_0,b_1, \ldots, b_z$ and $B_0,
 B_1,\ldots,B_z$, respectively. Now let $r,s \in \{0,1,\ldots,z\}$, and consider the $Q$-adic expansions 
 $$re = \sum_{t=1}^n Q^{t-1}I_t \quad \hbox{and} \quad se = \sum_{t=1}^n Q^{t-1}J_t$$
 ($I_t,J_t \in \{0,1,\ldots,Q-1\}$ for $t=1,\ldots,n$). Then $I_t$ and $J_t$ have $q$-adic expansions
 $$I_t = \sum_{u=1}^k q^{u-1}i_{tu} \quad \hbox{and} \quad J_t = \sum_{u=1}^k q^{u-1}j_{tu}$$
 ($i_{tu},j_{tu} \in \{0,1,\ldots,q-1\}$ for $t=1,\ldots,n$ and $u=1,\ldots,k$). Thus the $q$-adic expansions of $re$ and $se$ are 
 $$re = \sum_{t=1}^n \sum_{u=1}^k q^{u-1+k(t-1)}i_{tu} \quad \hbox{and} \quad se = \sum_{t=1}^n \sum_{u=1}^k q^{u-1 + k(t-1)} j_{tu}.$$
 If $b_rb_s \neq 0$ then $i_{tu} + j_{tu} < q$ for all $t,u$.
But then also $I_t + J_t < Q$ for all $t$, i.e. $B_r B_s \neq 0$.
Similarly,
$b_{r_1} \cdots b_{r_v} \neq 0$ implies $B_{r_1} \cdots B_{r_v} \neq 0$.
This shows that $\LL(A(q,kn,e)) \leq \LL(A(Q,n,e))$.
\end{proof}

\begin{rem}
In general, $A(q,kn,e)$ and $A(q^k,n,e)$ are not isomorphic.
Computational experiments show that
if they have the same Loewy vector and if $z \leq 10\,000$ holds
then mapping the basis vectors $b_i$ to the corresponding basis vectors $B_i$
defines an isomorphism.
\end{rem}

Next we improve a little on
\cite[Theorem~\refPartI{thm_upper_and_lower_bounds_I}~(ii)]{BHHK1}.

\begin{prop}\label{prop_improved_lower_bound}
 Let $q,n,e$ be as usual, and let $\nu := \ord_e(q)$. If $m := m(q,e) \nmid q-1$ then $\LL(A(q,n,e)) \geq n \lfloor \frac{q-1}{m} \rfloor +
 \frac{n}{\nu} +1$.
\end{prop}

\begin{proof}
Let $r \in \N$ such that $n = \nu r$.
Then Proposition~\ref{prop_n1_plus_n2} and
\cite[Theorem~\refPartI{thm_upper_and_lower_bounds_I}~(ii)]{BHHK1} imply:
\begin{eqnarray*}
\LL(A(q,n,e)) &\geq &r \cdot \LL(A(q,\nu,e))-r+1 \geq r \left(\nu \left\lfloor \frac{q-1}{m} \right\rfloor + 2 \right)-r+1 \cr
&= &n \left\lfloor \frac{q-1}{m} \right\rfloor +r+1.
\end{eqnarray*}
\end{proof}

\begin{rem}\label{rem_upper_bound_equals_lower_bound}
If $m = m(q,e) \nmid q-1$ then,
by \cite[Theorem~\refPartI{thm_upper_and_lower_bounds_I}]{BHHK1}
and Proposition~\ref{prop_improved_lower_bound}, we have
 $$n \left\lfloor \frac{q-1}{m} \right\rfloor + \frac{n}{\nu} + 1 \leq \LL(A(q,n,e)) \leq \left\lfloor n \frac{q-1}{m} \right\rfloor + 1.$$
 Let $a,r \in \N_0$ such that $q-1 = am+r$ and $0 \leq r < m$. It is routine to check that the upper and lower bound for $\LL(A(q,n,e))$ 
 coincide if and only if $r < \frac{m}{\nu} + \frac{m}{n}$. Thus in this case we have $\LL(A(q,n,e)) = \lfloor n \frac{q-1}{m} \rfloor + 1$.
\end{rem}

Here is an example where $(x_1 x_2 \cdots x_n)^{q-1}$ can be
decomposed into $\lfloor n (q-1) / m\rfloor$ monomials in $\J( A(q,n,e) )$
but where no factor of degree $m$ can occur.
(There are not very many such examples of dimension up to $10\,000$,
and this is the example of smallest dimension.)

\begin{example}
Let $q = 55$ and $e = (q^n-1) / 123 = 680\,763\,722\,688$.
Then $n = \ord_e(q) = 8$ and $m = m(q,e) = 126$.

We have $(x_1 x_2 \cdots x_8)^{18} \in \J( A(q,n,e) )$,
which implies that $\LL( A(q,n,e) ) \geq 4$ holds.
Since this is equal to the upper bound $\lfloor n (q-1) / m \rfloor + 1$,
we have equality.

The monomials in $\J( A(q,n,e) )$ have the degrees (and multiplicities)
126 (8), 144 (1), 198 (32), 216 (40), 234 (32), 288 (1), 306 (8),
and 432 (1).
Thus monomials of degree larger than $144$ cannot occur
in a decomposition of $(x_1 x_2 \cdots x_8)^{q-1}$ into three factors.
Hence monomials of degree 126 cannot occur in such a decomposition, either.

This means that no monomial of degree $m$ can occur in a decomposition
of $(x_1 x_2 \cdots x_8)^{q-1}$ into three factors.
\end{example}

\begin{prop}\label{decomp_into_1}
Let $q,q',e,n \in \N$ such that $q>1$, $e \mid q^n-1$ and $q' \equiv q \pmod{m}$ where $m := m(q,e)$. 
If $\lfloor n (q'-1) / m \rfloor = 1$ then $\LL( A(q,n,e) ) = \lfloor n (q-1) / m \rfloor + 1$.
\end{prop}

\begin{proof}
Let $q = q' + a m$ for some nonnegative integer $a$.
We have to show that $(x_1 x_2 \cdots x_n)^{q-1} \in J = \J(A(q,n,e))$
is a product of
$\lfloor n (q-1) / m \rfloor = a n + 1$ monomials in $J$.
This follows from the facts that
$(x_1 x_2 \cdots x_n)^{am}$ is a product of $a n$ monomials in $J$
(take $a$ times the $n$ cyclic shifts of a monomial of degree $m$)
and that $(x_1 x_2 \cdots x_n)^{q'-1} \in J$.
\end{proof}

Our next result is similar to Remark~\ref{rem_ll_for_rn_from_ll_for_n}.

\begin{lemma}\label{lemma_multiples_of_n}
 Let $q,n,e \in \N$ such that $q>1$, $e>1$ and $e \mid q^n-1$. 
 Moreover, let $m = m(q,e)$, and suppose that $\LL(A(q,n,e)) = \lfloor n \frac{q-1}{m} \rfloor + 1$ and $n(q-1) \equiv 1 \pmod{m}$.
 Then $\LL(A(q,kn,e)) = \lfloor kn \frac{q-1}{m} \rfloor + 1$ for $k=1,\ldots,m-1$.
\end{lemma}

\begin{proof}
 Let $a \in \N_0$ such that $n(q-1) = am+1$, and let $k \in \{1,\ldots,m-1\}$. 
 Then $\lfloor n \frac{q-1}{m} \rfloor = \lfloor \frac{am+1}{m} \rfloor = a$ and $\lfloor kn \frac{q-1}{m} \rfloor = ka$. 
 Since $\LL(A(q,n,e)) = a+1$, the monomial $(x_1 \ldots x_n)^{q-1}$ can be written as a product of $a$ monomials in $\J(A(q,n,e))$.
 Thus the monomial $(x_1 \ldots x_{kn})^{q-1}$ can be written as a product of $ka$ monomials in $\J(A(q,kn,e))$. The result follows.
\end{proof}


\section{The case that $e$ is a prime power}\label{section_prime_power}

As before, let $q,n,e \in \N$ such that $q>1$ and $e \mid q^n-1$. We set $A := A(q,n,e)$, $J := \J (A)$ and $m := m(q,e)$.
In this section, we will give several conditions which imply that $\LL (A) = \lfloor n \frac{q-1}{m} \rfloor + 1$. 

\begin{lemma}\label{power_sum}
    If $m = m(q,e)$ divides $n / r$, for some divisor $r$ of $n$,
    and $e$ divides $\sum_{i=0}^{m-1} q^{n i/(m r)}$ then
    $\LL(A) = \lfloor n \frac{q-1}{m} \rfloor + 1$.
\end{lemma}

\begin{proof}
    Let $a = n/(m r)$.
    We have $x_1 x_{a+1} x_{2a+1} \cdots x_{(m-1)a+1} \in \J(A)$ of degree $m$.
    The product of $a$ cyclic shifts of this monomial yields
    $x_1 x_2 \cdots x_{ma} = x_1 x_2 \cdots x_{n/r}$,
    and $r$ cyclic shifts by $n/r$ positions of this product yield
    $x_1 x_2 \cdots x_n$.
    Thus we have found a decomposition of $(x_1 x_2 \cdots x_n)^{q-1}$
    into $a r (q-1) = n (q-1) / m$ factors in $\J(A)$.
    Hence
    \[
       \LL(A) \geq \frac{n (q-1)}{m} + 1
              = \left\lfloor \frac{n (q-1)}{m} \right\rfloor + 1
              \geq  \LL(A).
    \]
\end{proof}

We apply this in the proof of our next result.

\begin{prop}\label{ll_for_special_small_m}
If $e$ divides $q^k+1$ or $q^{2k}+q^k+1$ for some $k \in \N$ then 
$$\LL(A(q,n,e)) = \left\lfloor n \frac{q-1}{m} \right\rfloor + 1.$$
\end{prop}

\begin{proof}
Suppose first that $e \mid q^k+1$ for some $k \in \N$. Then also $e \mid q^{2k}-1$, and $2 \geq m(q,e)$.
Hence the result follows from \cite[Lemma~\refPartI{lem_LL_elementary}~(iv)]{BHHK1}.

Now suppose that $e \mid q^{2k}+q^k+1$ for some $k \in \N$,
and set $n = 3 k$.
Then $e \mid q^n-1$ and $m(q,e) \leq 3$.
By \cite[Lemma~\refPartI{lem_LL_elementary}~(iv)]{BHHK1},
we may assume that $m(q,e) = 3$.
Since $m$ divides $n$ and $e$ divides
$1 + q^k + q^{2k} = \sum_{i=0}^{m-1} q^{n i/m}$,
we can apply Lemma~\ref{power_sum}.
\end{proof}

Now we come to the proof of the main result of this section.

\begin{thm}\label{thm_ll_for_prime_powers}
We have
$\LL( A(q,n,e) ) = \left\lfloor \frac{n (q-1)}{m(q,e)} \right\rfloor + 1$
in each of the following cases.
\begin{itemize}
\item[(i)]
    $e = p^k$ for an odd prime $p$ and $k \in \N$,
    and $q \equiv 1 \pmod{p}$, or \\
    $e = 2 p^k$ for an odd prime $p$ and $k \in \N$,
    and $q \equiv 1 \pmod{2p}$;
\item[(ii)]
    $e$ is a power of an odd Pierpont prime
    (cf.~Remark~\ref{pierpont_prime_remark});
\item[(iii)]
    $e$ is a power of $2$.
\end{itemize}
\end{thm}

\begin{proof}
\ 

\begin{itemize}
\item[(i)]
    Proposition~\ref{prop_m_equals_e1} yields that $m(q,e) = \gcd(e,q-1)$
    in this case.
    Now apply \cite[Theorem~\refPartI{thm_upper_and_lower_bounds_I}~(iii)]{BHHK1}.
\item[(ii)]
    Let $p$ be the prime that divides $e$.
    Apply part~(i) in the cases where $q \equiv 1 \pmod{p}$,
    and Corollary~\ref{pierpont_prime_cor} and
    \cite[Lemma~\refPartI{lem_LL_elementary}]{BHHK1} in the cases
    where $\ord_e(q)$ is even.

    In the remaining cases, $\ord_e(q) = 3^c p^l$ with $c > 0$.
    Set $t:= \ord_e(q)/3$ and note that $e$ divides
    $q^{3t}-1 = (q^t-1) (1 + q^t + q^{2t})$.
    Since $\ord_p(q) = 3^c$ holds, $p$ does not divide $q^t-1$,
    thus $e$ divides $1 + q^t + q^{2t}$,
    and Proposition~\ref{ll_for_special_small_m} can be applied.
\item[(iii)]
    We follow the proof of Proposition~\ref{m_for_two_power}.
    Apply \cite[Theorem~\refPartI{thm_upper_and_lower_bounds_I}~(iii)]{BHHK1}
    if $q \equiv 1 \pmod{4}$,
    and \cite[Lemma~\refPartI{lem_LL_elementary}]{BHHK1}
    if $q \equiv -1 \pmod{e}$.

    In all other cases, we have $m(q,e) = 4$, $e = 2^k$ for some $k \in \N$,
    and $q \equiv -5^{2^r a} \pmod{e}$ for some $r \in \{ 0, 1, \ldots, k-2 \}$
    and some odd $a \in \N$.
    Let $Q \equiv -5^{2^r} \pmod{e}$.
    The proof of Proposition~\ref{m_for_two_power}~(iii) shows that
    $e$ divides $2Q^c + Q^d + 1$,
    for suitable $c, d \in \{ 0, 1, \ldots, \ord_e(q)-1 \}$,
    where $c$ is odd and $d$ is even.
    Since $\langle q + e \Z \rangle = \langle Q + e \Z \rangle$,
    there is an odd $t \in \Z$ with $q^t \equiv Q \pmod{e}$,
    which implies that $e$ divides $2q^{ct} + q^{dt} + 1$.
    Note that $ct$ is odd and $dt$ is even.
    Thus the $n/2$ cyclic shifts of $x_1 x_{dt+1} x_{ct+1}^2$
    by an even number of positions yield a decomposition of
    $(x_1 x_2 \cdots x_n)^2$ into monomials of degree $m(q,e)$
    in $\J( A(q,n,e) )$.
\end{itemize}
\end{proof}


\section{The case that $e$ divides $(q^n-1)/(q-1)$}\label{section_polynomials}

We keep the notation of the preceding sections.

\begin{prop}\label{prop_m_geq_n_minus_1}
Let $e$ be a divisor of $\frac{q^n-1}{q-1}$ (so that $m = m(q,e) \leq n$
by \cite[Lemma~\refPartI{prop_m_and_e1}]{BHHK1}).
Then $\LL(A) = \lfloor n \frac{q-1}{m} \rfloor + 1$ holds
in the following cases.

(i) $n-m \mid \lfloor (n-m) \frac{q-1}{m} \rfloor$,

(ii) $m \geq n-1 $,

(iii) $m = n-2$ and $q \equiv a \pmod{m}$ for some $a \in \{1,\ldots, \lfloor \frac{m+1}{2} \rfloor \}$.
\end{prop}

\begin{proof}
(i) We set $k := \lfloor (n-m) \frac{q-1}{m} \rfloor / (n-m)$ if $n > m$, and $k := 0$ if $n=m$. Then $0 \leq k \leq \frac{q-1}{m}$. Let $x_1^{i_1} \ldots x_n^{i_n}$
 be a monomial of degree $m$ in $A$. The product of the $n$ cyclic shifts of $x_1^{i_1} \ldots x_n^{i_n}$ is $x_1^m \ldots x_n^m$, and the $k$-th power of
 this monomial is $x_1^{km} \ldots x_n^{km}$. We multiply this with the $(q-1-km)$-th power of the monomial $x_1\ldots x_n \in A$ to obtain the nonzero monomial
 $x_1^{q-1} \ldots x_n^{q-1} \in A$, a product of
 $$kn + q - 1 - km = q-1+k(n-m) = \left\lfloor n \frac{q-1}{m} \right\rfloor $$
 monomials in $\J(A)$. Thus $\LL(A) \geq \lfloor n \frac{q-1}{m} \rfloor + 1$,
and the result follows from
\cite[Theorem~\refPartI{thm_upper_and_lower_bounds_I}]{BHHK1}.

(ii) This is an immediate consequence of (i).

(iii) We write $q = a+cm$ for some $c \in \N_0$. Then
$$\left\lfloor (n-m) \frac{q-1}{m} \right\rfloor = \left\lfloor 2 \frac{a-1+cm}{m} \right\rfloor = 2c + \left\lfloor 2 \frac{a-1}{m} \right\rfloor = 2c$$
is divisible by $2 = n-m$.
Thus the result follows again from (i).
\end{proof}

\begin{lemma}\label{lem_ll_for_large_N}
(i)
  Let $q, e, k \in\N$ such that $\gcd(q,e) = 1$ and $k<q$.
Moreover, let $n = \ord_e(q)$, and let $N$ be a multiple of $n$.
Then the monomial $(x_1 x_2 \cdots x_N)^k \in J = \J(A(q,N,e))$ is a product of
$M$ monomials $m_1, m_2, \ldots, m_M \in J$ if and only if
the monomial $(Y_1 Y_2 \cdots Y_n)^{kN/n}$ in the polynomial ring $F[Y_1,\ldots,Y_n]$ is a product of
$M$ monomials $p_1, p_2, \ldots, p_M$ with the property that
$p_i = Y_1^{c_{i,1}} \cdots Y_n^{c_{i,n}}$ such that
$\sum_{j=1}^n c_{i,j} q^{j-1} \equiv 0 \pmod{e}$ holds for $1 \leq i \leq M$.
Furthermore, the factors may be chosen such that the total degrees of $m_i$ and
$p_i$ are equal, for $1 \leq i \leq M$.

(ii)
 Let $q,e \in \N$ such that $q > 1$ and $\gcd(q,e) = 1$. Moreover, let $n := \ord_e(q)$, $m := m(q,e)$ and $m_1 := \gcd(m,q-1)$. 
 If $N$ is a multiple of $\frac{m}{m_1}n$ then $\LL(A(q,N,e)) = N \frac{q-1}{m} + 1$.
\end{lemma}
 
\begin{proof}
(i)
  If the factors $m_i$ are given by $m_i = x_1^{a_{i,1}} \cdots x_N^{a_{i,N}}$
  then define $p_i = Y_1^{c_{i,1}} \cdots Y_n^{c_{i,n}}$ by setting
  $c_{i,j} = \sum_{l=0}^{N/n-1} a_{i, j+ln}$.

  If the factors $p_i$ are given by $p_i = Y_1^{c_{i,1}} \cdots Y_n^{c_{i,n}}$
  then define $m_1, m_2, \ldots, m_M$ inductively:
  For $m_1$, distribute the exponents
  $c_{1,1}$ to $x_1, x_{n+1}, \ldots, x_{N-n+1}$,
  $c_{1,2}$ to $x_2, x_{n+2}, \ldots, x_{N-n+2}$,
  etc., such that all values are less than $q$.
  Then construct the exponent vector of $m_2$
  by distributing $c_{2,1}, \ldots, c_{2,n}$ such that the sum
  of the exponent vectors of $m_1$ and $m_2$ does not exceed $q-1$,
  and continue in this way.

(ii)
  It suffices to show that the monomial
  $(x_1 x_2 \cdots x_N)^{q-1} \in A(q,N,e)$ decomposes into a product of
  $\frac{N (q-1)}{m}$ factors in $\J(A(q,N,e))$.
  By part~(i), it suffices to show that $(Y_1 Y_2 \cdots Y_n)^{N(q-1)/n}$
  decomposes into a product of $\frac{N (q-1)}{m}$ admissible factors.
  Since $m$ divides $N(q-1)/n$,
  such a factorization is given by $N(q-1)/(mn)$ times the $n$ cyclic shifts
  of a monomial of degree $m$ in $A(q,n,e)$.
\end{proof}

Note that Lemma~\ref{lem_ll_for_large_N}~(ii) generalizes part
of \cite[Theorem~\refPartI{thm_upper_and_lower_bounds_I}]{BHHK1}.

\begin{prop}\label{prop_e_divides_Phi_d}
If $e \mid \frac{q^d-1}{q-1}$ for some $d \in \{2,3,4\}$ then 
$\LL(A(q,n,e)) = \left\lfloor n \frac{q-1}{m} \right\rfloor +1$.
\end{prop}

\begin{proof}
If $e \mid q+1$ then
the assertion follows from \cite[Lemma~\refPartI{lem_LL_elementary}]{BHHK1}. 

If $e \mid q^2+q+1$ then the assertion is a special case of Proposition~\ref{ll_for_special_small_m}.

Finally, suppose that $e \mid q^3 + q^2 + q + 1$.
Since $e \mid q^4-1$ and $e \mid q^n-1$,
we also have $e \mid q^g-1$ where $g := \gcd(4,n)$. 

If $g=1$ then $e \mid q-1$, and there is nothing to prove. 

If $g = 2$ then $e \mid \gcd(q^2-1,q^3+q^2+q+1) \mid 2q+2$.
Suppose first that $e$ is odd.
Then $e \mid q+1$, so that $m \leq 2$, and the result follows.
Suppose therefore that $e$ is even. Then $q$ is odd.
Moreover, $e_1 = \gcd(e,q-1)$ and $m$ are even.
Since $m \leq 4$ we may assume that $m=4$.
Then $x_1^2x_2^2 \in J := \J(A(q,n,e))$.
Since $n$ is also even we can write $x_1^{q-1} \ldots x_n^{q-1}$
as a product of $\frac{n}{2} \frac{q-1}{2}$ shifts of the
monomial $x_1^2x_2^2$.
The result follows in this case. 

This leaves the case $g=4$.
If $m=4$ then $x_1x_2x_3x_4 \in J$,
and we can write $x_1^{q-1} \ldots x_n^{q-1}$ as a product of
$x_1x_2x_3x_4$ and its cyclic shifts since $4 \mid n$.
The result follows in this case.

Thus we may assume that $m=3$.
In this case, $J$ contains one of the following monomials:
$$
x_1^3, x_1^2x_2,x_1^2x_3,x_1^2x_4, x_1x_2x_3, x_1x_2x_4, x_1x_3x_4.$$
In the first case, we have $e=3$,
and the result follows from
Table~\ref{table_meq_small} and \cite[Lemma~\refPartI{lem_LL_elementary}~(iv)]{BHHK1}.
In the next three cases, we have
$$
\begin{array}{l}
e \mid \gcd(q^3+q^2+q+1,q+2) \mid 5, \\
e \mid \gcd(q^3+q^2+q+1,q^2+2) \mid 3, \\
e \mid \gcd(q^3+q^2+q+1,q^3+2) \mid 5,
\end{array} $$
and we know that the claimed result holds.
In the last three cases, $e$ divides $1+q+q^2$, $1+q+q^3$ or $1+q^2+q^3$.
Since also $e \mid 1 + q + q^2 + q^3$ we obtain $e=1$,
and the result holds trivially.
\end{proof}

As before, for $d \in \N$,
we denote by $\Phi_d \in \Q[X]$ the $d$-th cyclotomic polynomial.

\begin{rem}\label{rem_e_divides_Phi_d}
Proposition~\ref{ll_for_special_small_m} implies:

If $e \mid \Phi_d(q)$ for some $d \in \{6,9,10\}$
or for a $2$-power $d$
then $\LL(A(q,n,e)) = \lfloor n \frac{q-1}{m} \rfloor +1$. 
\end{rem}

Our next result extends Proposition~\ref{prop_e_divides_Phi_d}
to the case $d = 5$.
First we deal with a special case.

\begin{example}\label{example_11}
Let $e = 11$.
We show the equality
$\LL(A(q,n,e)) = \lfloor n \frac{q-1}{m(q,e)} \rfloor + 1 \quad (\ast)$.

If $q \equiv k \pmod{e}$ with $k \notin \{ 3, 4, 5, 9 \}$
then we are done by \cite[Remark~\refPartI{rem_bound_overview}]{BHHK1} and
\cite[Lemma~\refPartI{lem_LL_elementary}]{BHHK1},
so we can assume $k \in \{ 3, 4, 5, 9 \}$.
Then $m = 3$ and $\ord_e(q) = 5$.
We apply Proposition~\ref{prop_reduction_ll_for_n} with $N = 15$.
Thus it suffices to prove $(\ast)$ for $n \in \{5,10,15\}$.
By Lemma~\ref{lem_ll_for_large_N}~(ii),
$(\ast)$ holds for $n = 15$. Thus we may assume $n \in \{5,10\}$.

By Remark~\ref{rem_reduction_for_fixed_e_to_finitely_many_q},
it suffices to prove $(\ast)$ for $q \leq \lcm(e,m) = 33$.
For two such values of $q$ that differ by a multiple of $m$,
only the smaller one has to be verified,
by Remark~\ref{rem_Q_congruent_q_mod_m}.
Thus we may assume that $q \in \{3,4,5\}$.
If $q = 4$ then $m \mid q-1$, so that $(\ast)$ holds
by \cite[Theorem~\refPartI{thm_upper_and_lower_bounds_I}~(iii)]{BHHK1}.

Suppose that $n=5$. If $q=5$ then $(\ast)$ holds by Proposition~\ref{decomp_into_1}.
If $q=3$ then $x_1^2\ldots x_5^2$ is the product of the 3 monomials $x_1^2x_3$, $x_2^2x_4$ and $x_3x_4x_5^2$ in $\J(A(3,5,11))$. 
Thus $(\ast)$ holds in this case.
Hence we may now assume that $n = 10$ and $q \in \{3,5\}$.

If $q = 3$ then $x_1^2 \ldots x_{10}^2$ is the product of the 6 monomials
$x_1^2x_3$, $x_2^2x_4$, $x_3x_4x_5^2$, $x_6^2x_8$,
$x_7^2x_9$, $x_8x_9x_{10}^2$ in $\J(A(3,10,11))$.
Thus $\LL(A(3,10,11)) > 6 = \lfloor 10 \frac{2}{3} \rfloor$.

If $q = 5$ then $x_1^3 \ldots x_{10}^3$ is the product of the 10 cyclic shifts
of a monomial of degree $m = 3$ in $J := \J(A(5,10,11))$,
and $x_1 \ldots x_{10}$ is the product of the 3 monomials
$x_1x_5x_6$, $x_3x_7x_8$ and $x_2x_4x_9x_{10}$ in $J$.
Thus $x_1^4 \ldots x_{10}^4$ is a product of 13 monomials in $J$.
Hence $\LL(A(5,10,11)) > 13 = \lfloor 10 \frac{4}{3} \rfloor$.
\end{example}

\begin{prop}\label{prop_e_divides_Phi_5}
If $e$ divides $\Phi_d(q)$ for some $d \in \{1,2,3,4,5\}$ then 
$$\LL(A(q,n,e)) = \left\lfloor n \frac{q-1}{m} \right\rfloor + 1.$$
\end{prop}

\begin{proof}
By Proposition~\ref{prop_e_divides_Phi_d}, we may assume that $d=5$.
Since $e \mid \Phi_5(q) = \frac{q^5-1}{q-1}$,
\cite[Lemma~\refPartI{prop_m_and_e1}]{BHHK1} implies that $m \leq 5$.
By \cite[Lemma~\refPartI{lem_LL_elementary}~(iv)]{BHHK1},
we may also assume that $m \geq 3$, i.e. $m \in \{3,4,5\}$.
Moreover, by
\cite[Theorem~\refPartI{thm_upper_and_lower_bounds_I}]{BHHK1},
we may assume that $q \not\equiv 1 \pmod{m}$.
Furthermore,
by \cite[Remark~\refPartI{rem_bound_overview}]{BHHK1},
we may assume that $q \not\equiv 1 \pmod{e}$. Since $q^5-1 = (q-1) \Phi_5(q) \equiv 0 \pmod{e}$ this implies that $\nu := \ord_e(q) = 5$; in particular,
we have $5 \mid n$. Now we discuss the three possibilities for $m$ separately.

(i) Suppose first that $m=5$. Since $e_1 := \gcd(e,q-1) \mid m = 5$
by \cite[Lemma~\refPartI{prop_m_and_e1}]{BHHK1}, this implies $e_1 \in \{1,5\}$.
Since the assumption $e_1=5$ would lead to
the contradiction $m = 5 = e_1 \mid q-1$ we must have $e_1 = 1$. Then $\LL(A(q,5,e)) = \lfloor 5 \frac{q-1}{m} \rfloor + 1$
by \cite[Theorem~\refPartI{thm_upper_and_lower_bounds_I}~(iii)]{BHHK1}.
But then
Remark~\ref{rem_ll_for_rn_from_ll_for_n} implies that
$\LL(A(q,n,e)) = \lfloor n \frac{q-1}{m} \rfloor + 1$ whenever $5 \mid n \in \N$. 

(ii) Suppose next that $m=3$.
Then $e = 11$ by Example~\ref{example_Phi5}.
By 
Example~\ref{example_11},
the result follows in this case.

(iii) Thus we are left with the case $m=4$.
Then $e = 61$ by Example~\ref{example_Phi5}.
Thus $\ord_e(q) = 5$ implies that $q \equiv a \pmod{61}$ for some $a \in \{9,20,34,58\}$
(cf. Table~\ref{table_meq_large}).
By Remark~\ref{rem_reduction_for_fixed_e_to_finitely_many_q}, we may assume that $q = a + 61k$ for some $k \in \{0,1,2,3\}$. 
By Proposition~\ref{prop_reduction_ll_for_n}, we may also assume that $n \in \{5,10,15,20\}$.

If $n=5$ then Proposition~\ref{prop_m_geq_n_minus_1} implies the result.
If $n = 20$ then the result holds by Remark~\ref{rem_ll_for_rn_from_ll_for_n}.
Thus it remains to prove the result for $A(q,10,61)$ and $A(q,15,61)$.

Our hypothesis $q \not\equiv 1 \pmod{m}$ implies that we can ignore the cases $q \in \{9,81,217,241\}$. 

For $q \in \{34,58,70,142\}$ we have $5(q-1) \equiv 1 \pmod{4}$. In this case the (known) result for $n=5$ implies the result for $n=10$ and $n=15$,
by Proposition~\ref{prop_n1_plus_n2}.

Suppose that $q \in \{95,119,131,203\}$.
If $n=10$ then the result follows from Lemma~\ref{lem_ll_for_large_N},
and Corollary~\ref{cor_n1_plus_n2} implies the result for $n=15$.

It remains to deal with the cases $q \in \{20,156,180,192\}$ and $n \in \{10,15\}$. In these cases we have $q \equiv 0 \pmod{4}$ and write 
$$x_1^{q-1} \cdots x_n^{q-1} = x_1^{q-4} \cdots x_n^{q-4} \cdot x_1^3 \cdots x_n^3.$$
As usual, we can write $x_1^{q-4} \cdots x_n^{q-4}$ as a product of $n \frac{q-4}{4}$ monomials of degree $4$ in $A$. Thus it suffices to write $x_1^3 \cdots x_{10}^3$
as a product of $\lfloor  \frac{30}{4} \rfloor = 7$ monomials in $\J(A)$, and to write $x_1^3 \cdots x_{15}^3 $ as a product of $\lfloor \frac{45}{4} \rfloor = 11$
monomials in $\J(A)$.
Now the proof of Lemma~\ref{lem_ll_for_large_N} shows
that it suffices to write $x_1^6 \cdots x_5^6$ as a product of $7$ monomials in $\J(A(q,5,61))$,
and $x_1^9 \cdots x_5^9$ as a product of $11$ monomials in $\J(A(q,5,61))$.

Now observe that $x_1^6 \cdots x_5^6 = x_1^4 \cdots x_5^4 \cdot x_1^2 \cdots x_5^2$ where $x_1^4 \cdots x_5^4$ is a product of $5$ monomials of degree $4$ in $A(q,5,61)$,
and $x_1^2 \cdots x_5^2$ is a product of $2$ monomials of degree $5$ in $A(q,5,61)$. 

Similarly, we have $x_1^9 \cdots x_5^9 = x_1^8 \cdots x_5^8 \cdot x_1 \cdots x_5$ where $x_1^8 \cdots x_5^8$ is a product of $10$ monomials of degree $m=4$ in $A(q,5,61)$.
This finishes the proof of the proposition. 
\end{proof}

\begin{prop}
Let $n$ be a prime number. Then there are at most finitely many $e \in \N$ such that $e \mid \Phi_n(q)$ and $\LL(A(q,n,e)) \leq \lfloor n 
\frac{q-1}{m} \rfloor$ for some $q \in \N$.
\end{prop} 

\begin{proof}
If $q,e \in \N$ satisfy $q > 1$, $e \mid \Phi_n(q)$ and $m := m(q,e) = n$ then $\LL(A(q,n,e)) = \lfloor n \frac{q-1}{m} \rfloor + 1$ by 
Proposition~\ref{prop_m_geq_n_minus_1}.
Thus the result follows
from Proposition~\ref{prop_e_divides_Phi_n_and_m_smaller_than_n}.
\end{proof}

For $d=7$, Example~\ref{example_Phi5} gives a list of pairs $(m,e)$. 

\begin{lemma}\label{lemma_loewy_length_three}
 Let $n$ be even, and let $e$ be a proper divisor of $q^n-1$ and a multiple of $q^{\frac{n}{2}}-1$. Then $s_q(ke) = n \frac{q-1}{2}$ for $k = 1,\ldots, z-1$;
 in particular, we have $m(q,e) = n \frac{q-1}{2}$ and $\LL(A(q,n,e)) = 3 = n \frac{q-1}{m(q,e)} + 1$. 
\end{lemma}

\begin{proof}
 Let $r \in \{1,\ldots, q^{\frac{n}{2}}\}$, and consider the $q$-adic expansion $\sum_{i=0}^{\frac{n}{2}-1} a_iq^i$ of $q^{\frac{n}{2}} - r$. 
 Then $\sum_{i=0}^{\frac{n}{2}-1} a_iq^i + \sum_{i=0}^{\frac{n}{2}-1} (q-1-a_i)q^{\frac{n}{2} + i}$ is the $q$-adic expansion of $r(q^{\frac{n}{2}}-1)$. 
 Thus $s_q(r(q^{\frac{n}{2}}-1)) = \frac{n}{2}(q-1)$; in particular, $m(q,e) = \frac{n}{2}(q-1)$ and $\LL(A(q,n,e)) \leq \lfloor n \frac{q-1}{m(q,e)} \rfloor + 1 = 3$. 
 Since $e \neq q^n-1$ we conclude that $\LL(A(q,n,e)) = 3$.
\end{proof}


\section{Small values of $e$}

Let $e \in \N$ be fixed.
Then Proposition~\ref{prop_reduction_ll_for_n}
and Remark~\ref{rem_reduction_for_fixed_e_to_finitely_many_q} imply that,
in order to show that 
$$\LL(A(q,n,e)) = \left\lfloor n \frac{q-1}{m} \right\rfloor + 1$$
for all $q,n \in \N$ with $q>1$ and $e \mid q^n-1$ it suffices to check
a finite number of pairs $(q,n)$.

\begin{prop}\label{prop_e_at_most_32}
If $e \leq 32$ then
$\LL(A) = \lfloor n \frac{q-1}{m} \rfloor + 1 \quad (\ast)$.
\end{prop}

\begin{proof}
Consider the set of all pairs $(e,q+e\Z)$ with $e,q \in \N$,
$e \leq 32$ and $\gcd(q,e) = 1$. 
If $q \equiv 1 \pmod{e}$ then $(\ast)$ holds,
by \cite[Remark~\refPartI{rem_bound_overview}~(viii)]{BHHK1}.
Thus we can eliminate the corresponding pairs $(e,q+e\Z)$.
If $e$ is a power of 2 or a power of an odd Pierpont prime then $(\ast)$ also holds,
by Theorem~\ref{thm_ll_for_prime_powers}.
Thus we can also remove the corresponding pairs.
Similarly, we can eliminate the pairs where $m(q,e) = 2$,
by \cite[Lemma~\refPartI{lem_LL_elementary}~(iv)]{BHHK1}.
Also, by \cite[Theorem~\refPartI{thm_upper_and_lower_bounds_I}~(iii)]{BHHK1},
we can remove the pairs $(e,q+e\Z)$ where $m = e_1 := \gcd(e,q-1)$. 
If $e$ divides $\Phi_d(q)$ for some $d \in \{1,2,3,4,6,8,9,10\}$, 
or if $e$ divides $q^3+q^2+q+1$ or $q^4 +q^2 + 1$ or $q^6+q^3+1$
or $q^{10} + q^5 + 1$
then $(\ast)$ holds,
by the results in Section~\ref{section_prime_power} and Section~\ref{section_polynomials}.
Hence we can also eliminate these pairs. 
The case $e = 11$ has been treated in Example~\ref{example_11}.
The remaining pairs $(e,q+e\Z)$ are given by the following table:

\begin{center}
\footnotesize
\begin{tabular}{c|c|c|c|c|c}
$e$ & $14$ & $15$ & $21$ & $22$ & $23$                          \cr \hline
$q$ & $9,11$ & $4$ & $13$ & $3,5,9,15$ & $2,3,4,6,8,9,12,13,16,18$  \cr 
\end{tabular}
\end{center}

\begin{center}
\footnotesize
\begin{tabular}{c|c|c|c|c|c}
$e$ & $24$ & $26$ & $28$ & $29$ & $31$                      \cr \hline
$q$ & $5$ & $3,9$ & $11,23$ & $7,16,20,23,24,25$ & $2,4,8,16$   \cr 
\end{tabular}
\end{center}

Now consider a pair $(e,q+e\Z)$ where $\ord_e(q) \leq 3$
and $m \mid \ord_e(q)(q-1+er)$ for all $r \in \N$.
Since $(\ast)$ holds for $n \leq 3$ by \cite[Corollary~\refPartI{cor_n_leq_3}]{BHHK1},
$(\ast)$ holds for all admissible $n$,
by Proposition~\ref{prop_reduction_ll_for_n}.
By this argument, we can eliminate the pairs
$(15,4+15\Z)$, $(21,13+21\Z)$, $(24,5+24\Z)$, $(26,3+26\Z)$, and $(26,9+26\Z)$.

Next consider the pair $(14,9+14\Z)$.
Then $m=4$ and $\ord_{14}(9)=3$.
By Remark~\ref{rem_reduction_for_fixed_e_to_finitely_many_q},
it suffices to prove $(\ast)$
for $q \leq \lcm(e,m) = 28$, i.~e., for $q \in \{9,23\}$.
By \cite[Theorem~\refPartI{thm_upper_and_lower_bounds_I}~(iii)]{BHHK1},
$(\ast)$ holds for $q=9$ since $m \mid q-1$ in this case.
Thus we may assume that $q=23$.
Now, by an application of Proposition~\ref{prop_reduction_ll_for_n} with $N=6$,
we may assume that $n \in \{3,6\}$.
If $n=3$ then $(\ast)$ holds by \cite[Corollary~\refPartI{cor_n_leq_3}]{BHHK1}.
Thus we may assume that $n=6$.
But now $(\ast)$ holds, by Lemma~\ref{lem_ll_for_large_N}~(ii).

In a similar way, we can eliminate the pair $(14,11+14\Z)$.

Let $e=22$.
In the relevant cases, we have $m=4$ and $\ord_e(q) = 5$.
By Proposition~\ref{prop_reduction_ll_for_n} (with $N=10$)
we may assume that $n \in \{5,10\}$.
By Remark~\ref{rem_reduction_for_fixed_e_to_finitely_many_q},
we may also assume that $q \leq \lcm(e,m) = 44$,
i.~e., $q \in \{3,5,9,15,25,27,31,37\}$.
By \cite[Theorem~\refPartI{thm_upper_and_lower_bounds_I}]{BHHK1},
we may further assume that $ m \nmid q-1$,
i.~e., $q \in \{3,15,27,31\}$.
Then $m_1 := \gcd(m,q-1) = 2$.
Hence, by Lemma~\ref{lem_ll_for_large_N}, we may assume that $n=5$.

For $q = 3$, we need to write $(x_1 x_2 \ldots x_5)^2$ as a product of
$2$ monomials in $J := \J(A(q,n,e))$,
and $(x_1^2 x_4 x_5)(x_2^2 x_3^2 x_4 x_5)$ is such a decomposition.
For $q \in \{ 15, 27, 31 \}$,
we need to write $(x_1 x_2 \ldots x_5)^{q-1}$ as a product of
$\lfloor 5 \frac{q-1}{4} \rfloor = \frac{5q-7}{4}$ monomials
in $J := \J(A(q,n,e))$.
As usual, we can write $(x_1 x_2 \ldots x_5)^4$ as a product of
the $5$ cyclic shifts of a monomial of degree $4$ in $J$.
Thus we can write $x_1^{q-3} \ldots x_5^{q-3}$ as a product of
$5 \frac{q-3}{4} = \frac{5q-15}{4}$ monomials of degree 4 in $J$.
Hence it suffices to write $(x_1 x_2 \ldots x_5)^2$ as a product of
$2$ monomials in $J$,
which can be constructed from the above decomposition for $q = 3$
with the help of Proposition~\ref{prop_factorization_into_k_monomials}.

Let $e=23$.
In the relevant cases, we have $m=3$ and $\ord_e(q) = 11$.
We may assume that $2 \leq q < e m = 69$ and $n \in \{ 11, 22 \}$.
For two such $q$ that differ by a multiple of $m$, only the smaller one
must be verified, thus we have to consider only $q \in \{ 2, 3, 4 \}$.
Moreover, we can ignore the case $q=4$ since then $m$ divides $q-1$.
\begin{itemize}
\item
    $q = 2$, $n = 11$:
    Write
    $x_1 \cdots x_{11} =
    x_1 x_3 x_7 \cdot x_2 x_4 x_8 \cdot x_5 x_6 x_9 x_{10} x_{11}$.
\item
    $q = 2$, $n = 22$:
    Write
    $x_1 \cdots x_{22} =
    x_1 x_3 x_7 \cdot x_2 x_4 x_8 \cdot x_9 x_{11} x_{15} \cdot
    x_{10} x_{12} x_{16} \cdot x_5 x_{14} x_{20} \cdot
    x_{13} x_{17} x_{22} \cdot x_6 x_{18} x_{19} x_{21}$.
\item
    $q = 3$, $n = 11$:
    A decomposition of $x_1 \cdots x_{22}$ into $7$ monomials as for $q = 2$
    exists also for $q = 3$,
    by Proposition~\ref{prop_factorization_into_k_monomials},
    and can be turned into a decomposition of
    $(x_1 \cdots x_{11})^2$ into $7$ monomials,
    as in Lemma~\ref{lem_ll_for_large_N}~(i).
\item
    $q = 3$, $n = 22$:
    In this case Lemma~\ref{lemma_multiples_of_n} gives the result.
\end{itemize}

Let $e=29$.
For the relevant values of $q$, we have $m(q,e) = 4$ and $\ord_e(q) = 7$,
thus all relevant values of $q$ generate the same group of residues
modulo $e$.

It is enough to verify $(\ast)$ for $2 \leq q < e m = 116$
and $n \in \{ 7, 14, 21 \}$.
For two such $q$ that differ by a multiple of $m$, only the smaller one
must be verified, thus we have to consider only $q$ congruent to one of
$7, 16, 25$ modulo $e$.
This leaves $q \in \{ 7, 16, 25, 54 \}$ to be considered.
Moreover, we can ignore the case $q = 25$ since then $m$ divides $q-1$.

\begin{itemize}
\item
    For $q = 7$,
    $n = 7$ is done by
    decomposing
    \[
      (x_1 x_2 \cdots x_7)^3 = (x_1^2 x_2 x_3) (x_2^2 x_3 x_4)
        (x_4^2 x_5 x_6) (x_5^2 x_6 x_7) (x_1 x_3 x_6 x_7^2),
    \]
    which establishes a decomposition of 
    $(x_1 x_2 \cdots x_7)^6$ into $10$ invariant monomials, and
      $n = 14$ is done by Lemma~\ref{lem_ll_for_large_N}.
    Moreover, $n = 21$ need not be considered by Corollary~\ref{cor_n1_plus_n2}.
\item
    Let $q = 16$.
    For $n = 7$, 
    Proposition~\ref{prop_factorization_into_k_monomials} and the above
    decomposition of $(x_1 x_2 \cdots x_7)^3$ for $q = 7$
    yield the required decomposition of $(x_1 x_2 \cdots x_7)^{15}$,
    using the generic decomposition of $(x_1 x_2 \cdots x_7)^{12}$ into
    cyclic shifts of a monomial of minimal degree.
    The cases $n = 14$ and $n = 21$ follow by Lemma~\ref{lemma_multiples_of_n}.
\item
    Let $q = 54$.
    For $n = 7$, Proposition~\ref{decomp_into_1} strikes.
    For $n = 14$, write
    \[
       (x_1 x_2 \cdots x_{14})^1 = (x_1 x_4 x_7 x_8) (x_2 x_3 x_6 x_{10})
          (x_9 x_{11} x_5 x_{12} x_{13} x_{14}).
    \]
    For $n = 21$,
    construct a decomposition of $(x_1 x_2 \cdots x_7)^3$ into $5$ monomials
    from one for $q = 7$ and $n = 7$,
    again using Proposition~\ref{prop_factorization_into_k_monomials},
    and distribute it to $n = 21$ as in Lemma~\ref{lem_ll_for_large_N}~(i).
\end{itemize}

Suppose that $e = 31$.
In each case, $m = 5 = \ord_e(q)$.
By Remark~\ref{rem_ll_for_rn_from_ll_for_n},
it suffices to prove $(\ast)$ for $n=5$.
Since $e \mid \frac{q^5-1}{q-1}$,
Proposition~\ref{prop_m_geq_n_minus_1} implies $(\ast)$. 

Suppose that $e = 28$ and $q+e\Z \in \{11+e\Z, 23+e\Z\}$.
Then $m=4$ and $\ord_e(q) = 6$. 
Thus, by Remark~\ref{rem_ll_for_rn_from_ll_for_n},
it suffices to prove $(\ast)$ for $n=6$.
If $q \equiv 11 \pmod{e}$ (resp. $q \equiv 23 \pmod{e}$) then
$x_1^2\ldots x_6^2$ is the product of the 3 monomials
$x_1^2x_2x_4$, $x_3^2x_4x_6$ and $x_2x_5^2x_6$
(resp. $x_1x_3x_4^2$, $x_3x_5x_6^2$ and $x_1x_2^2x_5$)
in $J := \mathrm{J}(A(q,6,e))$.
Thus $x_1^{q-1} \ldots x_6^{q-1}$ is the product of $3 \frac{q-1}{2}$ elements
in $J$, so that $\LL(A(q,6,e)) > \lfloor 6 \frac{q-1}{m} \rfloor$. 
\end{proof}

The following result establishes an infinite series of examples $A(q,n,e)$, with $e = 33$,
for which the upper bound on the Loewy length from
\cite[Theorem~\refPartI{thm_upper_and_lower_bounds_I}]{BHHK1} is not attained.

\begin{prop}\label{prop_small_LL_e=33}
Let $(q,e) = (5,33)$.
Then $m = m(q,e) = 3$,
and $\LL( A(q,n,e) ) = \lfloor \frac{(q-1)n}{m} \rfloor + \epsilon$,
where $\epsilon = 0$ if $n \equiv 10 \pmod{30}$,
and $\epsilon = 1$ otherwise.
\end{prop}

\begin{proof}
We have $\ord_e(q) = 10$,
$m = m(q,e) > 2$ because $q^5 \equiv 23 \not\equiv -1 \pmod{e}$,
and $2 + q^4 = 627 = 19 e$ establishes $m = 3$.

Let $n = 10$.
The values $(1 + q^i + q^j) \pmod{e}$, for $0 \leq i \leq j < n$,
are as follows.
\[
\begin{array}{r|rrrrrrrrrr}
   &  0 &  1 &  2 &  3 &  4 &  5 &  6 &  7 &  8 &  9 \\ \hline
 0 &  3 &  7 & 27 & 28 &  0 & 25 & 18 & 16 &  6 & 22 \\
 1 &    & 11 & 31 & 32 &  4 & 29 & 22 & 20 & 10 & 26 \\
 2 &    &    & 18 & 19 & 24 & 16 &  9 &  7 & 30 & 13 \\
 3 &    &    &    & 20 & 25 & 17 & 10 &  8 & 31 & 14 \\
 4 &    &    &    &    & 30 & 22 & 15 & 13 &  3 & 19 \\
 5 &    &    &    &    &    & 14 &  7 &  5 & 28 & 11 \\
 6 &    &    &    &    &    &    &  0 & 31 & 21 &  4 \\
 7 &    &    &    &    &    &    &    & 29 & 19 &  2 \\
 8 &    &    &    &    &    &    &    &    &  9 & 25 \\
 9 &    &    &    &    &    &    &    &    &    &  8
\end{array}
\]
We see that the only possibilities for $1 + q^i + q^j \equiv 0 \pmod{e}$
are $2 + q^4$ and $1 + 2 q^6$.
Thus all monomials of degree $m$ in $A = A(q,n,e)$ are $x_1^2 x_5$
and its cyclic shifts.

We show that $(x_1 x_2 \cdots x_n)^{q-1} \in A$ is a product of $12$
monomials in $A$ but not a product of $13$ such monomials.

A decomposition into $12$ factors
is given by decomposing $(x_1 x_2 \cdots x_n)^3$ into $10$ factors of
degree $3$ (taking each cyclic shift of $x_1^2 x_5$ with multiplicity one),
and then writing $x_1 x_2 \cdots x_{10}$ as a product of $x_1 x_4 x_6 x_7$
(since $1 + q^3 + q^5 + q^6 = 572 e$) and
$x_2 x_3 x_5 x_8 x_9 x_{10}$.

In any factorization of $(x_1 x_2 \cdots x_n)^{q-1}$,
each of the cyclic shifts of $x_1^2 x_5$ can appear with multiplicity
at most two.
For each monomial $x_i^2 x_{\overline{i+4}}$ with multiplicity two,
the monomial $x_{\overline{i+6}}^2 x_i$ cannot appear at all,
hence the number of cyclic shifts with multiplicity one is at most $10-2k$.
Thus the multiplicity of degree $3$ monomials in a factorization
is at most $2k + (10-2k) = 10$, which is too small for a factorization
into $13$ monomials.

Let $n = 20$.
We have to decompose $(x_1 \cdots x_{20})^4 \in A = A(q,n,e)$ into
$26$ monomials in $A$.
For that, decompose $(x_1 \cdots x_{20})^3$ into $20$ monomials (take the
cyclic shifts of a monomial of degree $3$) and decompose
$(y_1 \cdots y_{10})^2$ with the following exponent vectors:
\[
   \begin{array}{rrrrrrrrrr}
    (2, 0, 0, 0, 1, 0, 0, 0, 0, 0), \\
    (0, 2, 0, 0, 0, 1, 0, 0, 0, 0), \\
    (0, 0, 1, 0, 1, 1, 0, 0, 0, 1), \\
    (0, 0, 0, 2, 0, 0, 0, 1, 0, 0), \\
    (0, 0, 0, 0, 0, 0, 2, 1, 0, 1), \\
    (0, 0, 1, 0, 0, 0, 0, 0, 2, 0).
   \end{array}
\]

For $n = 30$, the maximal decomposition follows from
Lemma~\ref{lem_ll_for_large_N}.

Applying Corollary~\ref{cor_n1_plus_n2} to the above results
yields the claim for $n \equiv 0 \pmod{30}$ and $n \equiv 20 \pmod{30}$.

It remains to compute the Loewy length for $n \equiv 10 \pmod{30}$
and $n > 10$.
The above proof for $n = 10$ can be generalized, as follows.
Let $n = 30 a + 10$, for a nonnegative integer $a$.
We claim that the Loewy length of $A = A(5,n,33)$ is $40 a + 13$.
By Proposition~\ref{prop_n1_plus_n2}, this value is a lower bound:
Take $n_1 = 10, n_2 = 30 a$.
Thus we have to show that $(x_1 \cdots x_n)^4$ cannot be decomposed
into $40 a + 13$ factors in $A$.
As in the case $n = 10$,
we show that $(y_1 \cdots y_{10})^{4N/n}$ cannot be decomposed
into $40 a + 13$ allowed monomials.
Since $4N = 120 a + 40 = (40 a + 12) \cdot 3 + 1 \cdot 4$,
we would need $40 a + 12$ factors of degree $3$ for such a factorization.
The possible factors of degree $3$ are the $10$ cyclic shifts of $y_1^2 y_5$,
each with multiplicity at most $2N/n = 6 a + 2$.
Set $k_0 = 4 a + 1$.
If the multiplicity of a degree $3$ monomial $y_i^2 y_{\overline{i+4}}$
as a factor is $k > k_0$ then
the multiplicity of the monomial $y_{\overline{i+6}}^2 y_i$ is at most
$k_1 \leq 4N/n - 2k \geq 0$.
Thus
\[
   k + k_1 \leq 4N/n - k
           < 4N/n - k_0 = 12 a + 4 - (4 a + 1) = 8 a + 3
           \leq 2 k_0 + 1,
\]
which means $k + k_1 \leq 2 k_0$.
Thus we can consider the ten cyclic shifts in pairs,
and the total number of degree $3$ monomials is at most $10 k_0 < 40 a + 12$.
This is too small for a factorization into $40 a + 13$ monomials.
\end{proof}

\begin{rem}
Explicit computations show that the Loewy vector of $A(5,10,33)$ is
\[
   (1, 440, 4296, 17770, 42595, 66482, 71186, 53392, 27865, 9710,
   2011, 180, 1).
\]
\end{rem}



\begin{prop}\label{prop_LL_for_large_m_relative_to_e}
If $m = m(q,e) \geq e/3$ then
$\LL(A(q,n,e)) = \lfloor n \frac{q-1}{m} \rfloor +1$.
\end{prop}

\begin{proof}
As mentioned in the proof of Proposition~\ref{q_bigger_e},
we have $m = e_1 = \gcd(e, q-1)$ in the cases (i)--(iii),
and $e \leq 32$ in case~(iv).
Thus the claim follows from
\cite[Theorem~\refPartI{thm_upper_and_lower_bounds_I}]{BHHK1}
and Proposition~\ref{prop_e_at_most_32}.
\end{proof}



\begin{rem}
Fix $q \in \N$ such that $q > 1$.
By \cite[Theorem~\refPartI{thm_upper_and_lower_bounds_I}]{BHHK1},
we have
\[
   \LL(A(q,n,e)) \leq \left\lfloor n \frac{q-1}{m(q,e)} \right\rfloor + 1,
\]
and equality holds for $n \leq 3$,
by \cite[Corollary~\refPartI{cor_n_leq_3}]{BHHK1}.

In order to find the smallest $n$ such that some divisor $e$ of $q^n-1$
exists for which the above inequality is strict,
we may proceed as follows.

For increasing values of $n \geq 4$,
run through all divisors $e$ of $q^n-1$.
If none of the results from \cite{BHHK1} or from this paper implies
that equality holds for the triple $(q,n,e)$ then
explicitly compute $\LL(A(q,n,e))$
(for example using \cite[Proposition~\refPartI{prop_lambda}]{BHHK1})
and $m(q,e)$, and check.

Here is the list for $2 \leq q \leq 9$ which includes the algebra $A(5,10,33)$ of Proposition~\ref{prop_small_LL_e=33}.

\begin{center}
\begin{tabular}{c|r|l}
    $q$ & min. $n$ & $e$ such that the min. $n$ yields strict inequality \\
    \hline
    $2$ &     $20$ & $8\,525$ \\
    $3$ &     $12$ & $35, 1\,168, 7\,592$ \\
    $4$ &     $10$ & $275$ \\
    $5$ &     $10$ & $33$ \\
    $6$ &     $13$ & $3\,433$ \\
    $7$ &     $12$ & $2\,241\,504, 3\,735\,840, 23\,660\,320, 29\,139\,552,
                     70\,980\,960, 133\,089\,300$, \\
        &          & and perhaps others \\
    $8$ &     $12$ & $72\,412\,515, 278\,216\,505, 723\,362\,913$,
                     and perhaps others \\
    $9$ &      $9$ & $247$
\end{tabular}
\end{center}

Note that the ``brute force'' computation of the Loewy length
can be expensive for high dimensional algebras $A(q,n,e)$,
i.~e., small values of $e$.
We used a combination of programs in {\GAP} \cite{GAP} and
Julia \cite{Julia} for these computations.
\end{rem}

The following result gives more direct information concerning the Loewy length of $A(q,2,e)$.
However,
it seems to be difficult to prove similar results for $A(q,n,e)$
in case $n \geq 3$.

\begin{cor}
Let $e$ be a divisor of $q^2-1$,
and set $e_1 = \gcd(e,q-1)$, $e_2 = \gcd(e,q+1)$, and $A = A(q,2,e)$.
Then $\LL(A) = 2 \frac{q-1}{e_1} + 1$ if
$e_1 \geq e_2$ or both $e$ and $\frac{q^2-1}{e}$ are even,
and $\LL(A) = \frac{q-1}{e_1} + 1$ otherwise.
\end{cor}

\begin{proof}
Let $m = m(q,e)$.
We know from \cite[Corollary~\refPartI{cor_n_leq_3}]{BHHK1} that
$\LL(A) = \lfloor 2 \cdot \frac{q-1}{m} \rfloor + 1$ in the case $n = 2$.
Now apply Proposition~\ref{m_eq_e1}.
\end{proof}


\section{Small values of $z$}

Now we change our perspective,
and focus on $z = (q^n-1)/e$ instead of $e$.
For convenience, we introduce the notation
$A[q,n,z]$ and $m[q,n,z]$ for $A(q,n,e)$ and $m(q,e)$, respectively.

Let us fix a number $z$.
Since $z+1$ is the dimension of $A[q,n,z]$,
a finite set of parameters $(q, n)$ suffices
to cover all $A[q,n,z]$, up to isomorphism.

First we observe that only the smallest possible $n$
has to be considered,
which is the multiplicative order $\ord_z(q)$ of $q$ modulo $z$.

\begin{lemma}\label{lem_only_minimal_n}
If $z$ divides $q^n-1$ and $N$ is a multiple of $n$
then $A[q,n,z] \cong A[q,N,z]$.
\end{lemma}

\begin{proof}
By \cite[Theorem~\refPartI{thm_reduction_of_n}]{BHHK1},
\begin{eqnarray*}
    A[q,N,z] & = & A(q,N,(q^N-1)/z) =  A(q,N,(q^N-1)/(q^n-1) \cdot (q^n-1)/z) \\
             & \cong & A(q,n,(q^n-1)/z) = A[q,n,z].
\end{eqnarray*}
\end{proof}

\begin{rem}
Note that the upper bound from
\cite[Theorem~\refPartI{thm_upper_and_lower_bounds_I}~(i)]{BHHK1}
on the Loewy length of $A[q,N,z]$ is attained if and only if
it is attained for $A[q,n,z]$,
by \cite[Remark~\refPartI{rem_bound_overview}~(v)]{BHHK1}.
\end{rem}

\begin{example}\label{example_congruence_q_mod_z_pm_1}
\begin{enumerate}
\item
    By \cite[Corollary~\refPartI{cor_uniserial_case}]{BHHK1},
    $A[q,n,z]$ is uniserial if and only if
    $(q^n-1)/z$ is a multiple of $(q^n-1)/(q-1)$,
    that is, if $q \equiv 1 \pmod{z}$.
    In this case, Example~\ref{example_s_q_equal}~(i)
    shows that $s_q(ke) = k n \frac{q-1}{z}$, for $1 \leq k < z$.
\item
    Let $q \equiv -1 \pmod{z}$ and $z > 2$.
    Then $n$ is even, and we may choose $n = 2$, by 
    Lemma~\ref{lem_only_minimal_n}.
    As in Example~\ref{example_s_q_equal}~(ii),
    we have $s_q(ke) = q-1$, for $1 \leq k < z$.
    (If we admit larger values of $n$ then we get $s_q(ke) = n (q-1) / 2$.)
    Thus all monomials $b_k$, with $1 \leq k < z$,
    have the same degree and therefore belong to the same Loewy layer.
    In particular, we get $m[q,n,z] = n(q-1)/2$ and $\LL(A[q,n,z]) = 3$,
    which is equal to the upper bound
    $\lfloor n(q-1) / m[q,n,z] \rfloor + 1$.
\end{enumerate}
\end{example}

In the above examples, we have seen that the structure of
$A[q,n,z]$ depends only on the residue class of $q$ modulo $z$.
The following corollary will show that this holds in general,
which means that we have to consider only prime residues $q$ modulo $z$.
(As before, we replace $q = 1$ by $q = z+1$.)

For that, we need a technical lemma that describes,
in terms of residues modulo $z$,
whether the product of two basis vectors $b_k$, $b_l$
in $A[q,n,z]$ is zero.

\begin{lemma}\label{lem_carry_in_terms_of_z}
In the situation of Proposition~\ref{prop_coeffs_z},
there is a carry in the addition of the vectors of $q$-adic coefficients
of $k e$ and $l e$ if and only if
there is an index $i \in \{ 1, 2, \ldots, n \}$ such that
$\overline{k q^i} + \overline{l q^i} \geq z$ holds
and not all values 
$\overline{k q^j} + \overline{l q^j}$, $1 \leq j \leq n$,
are equal to $z$.
\end{lemma}

\begin{proof}
Set $c_{k,i} = \overline{k q^{n-i}}$, for $0 \leq i \leq n$,
as in the proof of Proposition~\ref{prop_coeffs_z}.
By this lemma, we know that a carry occurs if and only if
\[
   \frac{(c_{k,i} + c_{l,i}) q - (c_{k,i-1} + c_{l.i-1})}{z} \geq q
\]
holds for some $i \in \{ 1, 2, \ldots, n \}$.

In this case we have $(c_{k,i} + c_{l,i}) q \geq z q$,
and some $c_{k,j} + c_{l,j}$ is different from $z$ because otherwise
all coefficients of $k e + l e$ would be equal to $q-1$,
contradicting the assumption that a carry occurs.

Conversely,
assume that not all $c_{k,i} + c_{l,i}$ are equal
and that $c_{k,i} + c_{l,i} \geq z$ holds for some $i$.
Then we can choose $i \in \{ 1, 2, \ldots, n \}$
such that $c_{k,i} + c_{l,i} \geq z$ and
$c_{k,i-1} + c_{l,i-1} < c_{k,i} + c_{l,i}$.
(Start with the largest index $i$ for which $c_{k,i} + c_{l,i}$ is maximal,
and decrease $i$ until $c_{k,i-1} + c_{l,i-1} < c_{k,i} + c_{l,i}$ holds.
Since $c_{k,0} = c_{k,n}$ and $c_{l,0} = c_{l,n}$,
there is a positive index $i$ with the required property.)
Thus
\[
   (c_{k,i} + c_{l,i}) q - (c_{k,i-1} + c_{l.i-1})
   > (c_{k,i} + c_{l,i}) (q - 1)
   \geq z (q - 1)
\]
holds.
The left hand side is divisible by $z$, hence it is at least $z q$,
which means that there is a carry at $i$.
\end{proof}

\begin{lemma}\label{lem_same_cyclic_subgroup_of_residues}
Let $z$ be a positive integer and let $q$ and $Q$ be two prime residues
modulo $z$ that generate the same subgroup of order $n$, say,
in the group of prime residues modulo $z$.
Then $A[q,n,z]$ and $A[Q,n,z]$ are isomorphic.

In particular,
if $q \equiv Q \pmod{z}$ then $A[q,n,z] \cong A[Q,n,z]$.
\end{lemma}

\begin{proof}
We want to show that the two algebras have the same multiplication table
with respect to their natural bases.
We know that
\[
   \left\{ \overline{k q^i}; 0 \leq i \leq n-1 \right\} =
   \left\{ \overline{k Q^i}; 0 \leq i \leq n-1 \right\}
\]
holds for $1 \leq k < z$,
thus there is a permutation $\pi$ of $\{ 0, 1, \ldots, n-1 \}$
such that $q^t \equiv Q^{\pi(t)} \pmod{z}$ for $t = 0, \ldots, n-1$.
Hence the statement follows from Lemma~\ref{lem_carry_in_terms_of_z}:
The product of $b_k$ and $b_l$ in $A[q,n,z]$ is zero if and only if
there is an $i \in \{ 1, 2, \ldots, n \}$ such that
$\overline{k q^i} + \overline{l q^i} \geq z$ holds,
and that not all $\overline{k q^i} + \overline{l q^i}$ are equal to $z$.
Since
$\overline{k q^i} + \overline{l q^i} =
 \overline{k Q^{\pi(i)}} + \overline{l Q^{\pi(i)}}$,
this condition is satisfied if and only if it is satisfied for
$Q$ instead of $q$, and this holds if and only if the product of
$b_k$ and $b_l$ in $A[Q,n,z]$ is zero.
\end{proof}

\begin{rem}
In the situation of Lemma~\ref{lem_same_cyclic_subgroup_of_residues},
Proposition~\ref{prop_coeffs_z} yields that
the upper bound from
\cite[Theorem~\refPartI{thm_upper_and_lower_bounds_I}~(i)]{BHHK1}
for the Loewy length is the same for $A[q,n,z]$ and $A[Q,n,z]$.
\end{rem}

\begin{example}\label{example_small_z}
\begin{enumerate}
\item
    For $z \in \{ 2, 3, 4, 6 \}$,
    only the cases $q \equiv \pm 1 \pmod{z}$ occur,
    which were handled in Example~\ref{example_congruence_q_mod_z_pm_1}.
\item
    Consider $z = 5$.
    The cases $q \equiv \pm 1 \pmod{z}$ are known,
    they yield the algebras $A[6,1,5] = A(6,1,1)$ and $A[4,2,5] = A(4,2,3)$.
    Lemma~\ref{lem_same_cyclic_subgroup_of_residues} tells us
    that we have the isomorphisms
    \[
      \begin{array}{lclclccc}
       A[6,1,5] & \cong & A[11,1,5] & \cong & A[16,1,5] & \cong & \cdots & \cong \\
       A(6,1,1) & \cong & A(11,1,2) & \cong & A(16,1,3) & \cong & \cdots & \ 
      \end{array}
    \]
    and
    \[
      \begin{array}{lclclccc}
       A[4,2,5] & \cong & A[9,2,5] & \cong & A[14,2,5] & \cong & \cdots & \cong \\
       A(4,2,3) & \cong & A(9,2,16) & \cong & A(14,2,65) & \cong & \cdots & \ 
      \end{array}
    \]
    and also that the only other cases are
    \[
      \begin{array}{lclclccc}
       A[2,4,5] & \cong & A[7,4,5] & \cong & A[12,4,5] & \cong & \cdots & \cong \\
       A(2,4,3) & \cong & A(7,4,480) & \cong & A(12,4,4147) & \cong & \cdots & \ 
      \end{array}
    \]
    and
    \[
      \begin{array}{lclclccc}
       A[3,4,5] & \cong & A[8,4,5] & \cong & A[13,4,5] & \cong & \cdots & \cong \\
       A(3,4,16) & \cong & A(8,4,819) & \cong & A(13,4,5712) & \cong & \cdots & \ 
      \end{array}
    \]
    The latter two algebras, $A[2,4,5]$ and $A[3,4,5]$,
    are isomorphic by Lemma~\ref{lem_same_cyclic_subgroup_of_residues},
    and have Loewy length $3$.
    Note that $s_q(ke) = \frac{q-1}{z} \sum_{i=1}^4 \overline{k q^i}$,
    and $\{ \overline{k q^i}; 1 \leq i \leq 4 \} = \{ 1, 2, 3, 4 \}$
    is the set of all prime residues modulo $z$,
    for any $k \in \{ 1, 2, 3, 4 \}$;
    thus $s_q(ke) = 2 (q-1)$.
\item
    Consider $z = 7$.
    The cases $q \equiv \pm 1 \pmod{z}$ are known.
    The same arguments as in the case $z = 5$ show that
    $q \in \{ 3, 5 \}$ yields two isomorphic algebras $A[q,n,7]$
    of Loewy length $3$,
    because the sum of all prime residues modulo $z$ appears in
    the formula for $s_q(ke)$.

    The remaining cases are $q \equiv 2 \pmod{z}$ and $q \equiv 4 \pmod{z}$.
    Here the situation is different.
    We choose $n = \ord_z(q) = 3$ and compute $c_{k,i}$ and $a_{k,i}$;
    the rows in the following tables are indexed by $k$
    and the columns by $i$.

    (We know that it is sufficient to consider $q \in \{ 2, 4 \}$,
    and that the two values yield isomorphic algebras,
    but here we show the general case.)

    \[
       q \equiv 2 \pmod{z}:
       \ \ 
       \begin{array}{c|rrrr}
           c_{k,i} & 0 & 1 & 2 & 3 \\ \hline
                 1 & 1 & 4 & 2 & 1 \\
                 2 & 2 & 1 & 4 & 2 \\
                 3 & 3 & 5 & 6 & 3 \\
                 4 & 4 & 2 & 1 & 4 \\
                 5 & 5 & 6 & 3 & 5 \\
                 6 & 6 & 3 & 5 & 6
       \end{array}
       \ \ \ \ 
       \begin{array}{c|rrr}
           z a_{k,i} & 1 & 2 & 3 \\ \hline
                   1 & 4q-1 & 2q-4 &  q-2 \\
                   2 &  q-2 & 4q-1 & 2q-1 \\
                   3 & 5q-3 & 6q-5 & 3q-6 \\
                   4 & 2q-4 &  q-2 & 4q-1 \\
                   5 & 6q-5 & 3q-6 & 5q-3 \\
                   6 & 3q-6 & 5q-3 & 6q-5
       \end{array}
    \]
    \[
       q \equiv 4 \pmod{z}:
       \ \ 
       \begin{array}{c|rrrr}
           c_{k,i} & 0 & 1 & 2 & 3 \\ \hline
                 1 & 1 & 2 & 4 & 1 \\
                 2 & 2 & 4 & 1 & 2 \\
                 3 & 3 & 6 & 5 & 3 \\
                 4 & 4 & 1 & 2 & 4 \\
                 5 & 5 & 3 & 6 & 5 \\
                 6 & 6 & 5 & 3 & 6 \\
       \end{array}
       \ \ 
       \begin{array}{c|rrr}
           z a_{k,i} & 1 & 2 & 3 \\ \hline
                   1 & 2q-1 & 4q-1 &  q-4 \\
                   2 & 4q-2 &  q-4 & 2q-1 \\
                   3 & 6q-3 & 5q-6 & 3q-5 \\
                   4 &  q-4 & 2q-1 & 4q-2 \\
                   5 & 3q-5 & 6q-3 & 5q-6 \\
                   6 & 5q-6 & 3q-5 & 6q-3 \\
       \end{array}
    \]

    We see that the following holds in both cases:

    The exponent sums of $b_1, b_2, b_4$ are $q-1$,
    and the exponent sums of $b_3, b_5, b_6$ are $2(q-1)$.
    We have $b_3 = b_1 b_2$, $b_5 = b_1 b_4$, and $b_6 = b_2 b_4$.
    Since $m[q,3,z] = q-1$, the upper bound for the Loewy length
    of $A[q,3,z]$ is $4$, and thus $\LL(A[q,3,z]) = 4$ holds.
\end{enumerate}
\end{example}

We can generalize an observation from Example~\ref{example_small_z}.

\begin{prop}\label{prop_ll_for_maximal_order_of_q_modulo_prime_power}
Let $z$ be an odd prime power.
If $\ord_z(q) = \varphi(z)$ then we have 
$$\LL( A[q,\varphi(z),z] ) = 3.$$
\end{prop}

\begin{proof}
Apply Remark~\ref{rem_calculating_expansions}~(iii).
\end{proof}

\begin{prop}\label{lem_large_order_of_q_modulo_prime}
Let $z$ be an odd prime.
If $\ord_z(q) = (z-1) / 2$ then
\[
    \LL( A[q,(z-1)/2,z] ) = \left\{ \begin{array}{r@{\,,\quad}l}
             4 & \textrm{if $z \in \{ 3, 7 \}$} \\
             3 & \textrm{otherwise}
           \end{array} \right.
\]
\end{prop}

\begin{proof}
In the case $z \equiv 1 \pmod{4}$,
Remark~\ref{rem_calculating_expansions}~(iii) yields
$\LL( A[q,(z-1)/2,z] ) = 3$,
so assume $z \equiv -1 \pmod{4}$.
By Proposition~\ref{prop_large_order_of_q_modulo_prime},
exactly two different values occur for $s_q(ke)$, $1 \leq k \leq z-1$.
Thus $\LL( A[q,(z-1)/2,z] ) \in \{ 3, 4 \}$ holds,
and if the value is $4$ then $s_q(k e) = 2 s_q(e)$ must hold
for quadratic nonresidues $k$ modulo $z$,
which means that the sum $N$, say, of quadratic nonresidues modulo $z$
is twice as large as the sum $Q$, say, of quadratic residues modulo $z$.
This holds for $z \in \{ 3, 7 \}$,
and indeed we have $\LL( A[4,1,3] ) = \LL( A[2,3,7] ) = 4$.

The class number formula \cite[Chap.~6, equ.~(19)]{Davenport} (which has been
used in the proof of Proposition~\ref{prop_large_order_of_q_modulo_prime})
states that the ideal class number $h(-z)$ of the imaginary quadratic field
$\Q(\sqrt{-z})$ equals $(N-Q)/z$.
Since $N + Q = z(z-1)/2$ holds,
the condition $N = 2Q$ implies $h(-z) = (z-1)/6$.
hence it suffices to show that this equality cannot hold for primes $z$
with $z \equiv -1 \pmod{4}$ and $z > 7$.

For that, note that $h(-z) = \sqrt{z} L(1, \chi) / \pi$
(see \cite[Chap.~6, equ.~(15)]{Davenport}) for $z > 3$,
where $\chi$ is a primitive Dirichlet character modulo $z$ and $L(1, \chi)$
is the value of the Dirichlet $L$ function for $\chi$ at $1$.
Now \cite[Theorem~A]{Louboutin} (with $N = 0$)
states that $L(1, \chi) \leq 1 + \log(\sqrt{z})$,
which implies $h(-z) \leq \sqrt{z} ( 1 + \log(\sqrt{z}) ) / \pi$.
It is easy to check that $\sqrt{z} ( 1 + \log(\sqrt{z}) ) / \pi < (z-1)/6$
holds for $z > 27$,
and that $h(-z) \neq (z-1)/6$ for the relevant $z \in \{ 8, \ldots, 27 \}$.
\end{proof}

\begin{rem}
In the situation of
Proposition~\ref{prop_ll_for_maximal_order_of_q_modulo_prime_power},
the upper bound from
\cite[Theorem~\refPartI{thm_upper_and_lower_bounds_I}~(i)]{BHHK1}
is
\[
   \left\lfloor \frac{n(q-1)}{m[q,n,z]} \right\rfloor + 1 =
   \left\lfloor \frac{\varphi(z)(q-1)}{\varphi(z)(q-1)/2} \right\rfloor + 1 =
   3,
\]
hence it is equal to the Loewy length.
In the situation of Proposition~\ref{lem_large_order_of_q_modulo_prime},
we have
$m[q,n,z] = (z-1)(q-1)/4 = n(q-1)/2$ in the case $z \equiv 1 \pmod{4}$,
$m[q,n,z] = n(q-1)/3$ in the cases $z \in \{ 3, 7 \}$;
in the remaining cases, we have $m[q,n,z] = (q-1) Q / z$,
by the proof of Proposition~\ref{prop_large_order_of_q_modulo_prime},
and $N < 2Q$ implies $Q > z(z-1)/6 = z n/3$ and thus
$m[q,n,z] > n(q-1)/3$.
Together with the obvious inequality $m[q,n,z] \leq n(q-1)/2$,
we get that the upper bound is attained in each case.
\end{rem}

We can compute, for fixed $z$, the Loewy length of all $A[q,n,z]$.
The smallest value of $z$ for which this Loewy length differs
from the upper bound from
\cite[Theorem~\refPartI{thm_upper_and_lower_bounds_I}~(i)]{BHHK1}
is $z = 70$.

\begin{example}\label{expl_z_eq_70}
Let $n = 12$, $q = 3$, and $z = 70$; then $e = 7\,592$.
The exponent vectors of $b_1, b_2, \ldots, b_{z-1}$,
up to cyclic shifts, are as follows.
We list the value $k$ for which the shown vector belongs to $ke$,
the vector itself, the length of its orbit under cyclic shifts,
and $s_q(ke)$.
\[
\begin{array}{r|c|r|r}
     1 & [ 2, 1, 0, 2, 0, 1, 1, 0, 1, 0, 0, 0 ] & 12 &  8 \\
     2 & [ 1, 0, 1, 1, 1, 2, 2, 0, 2, 0, 0, 0 ] & 12 & 10 \\
     5 & [ 1, 2, 2, 1, 0, 0, 1, 2, 2, 1, 0, 0 ] &  6 & 12 \\
     7 & [ 2, 2, 0, 0, 2, 2, 0, 0, 2, 2, 0, 0 ] &  4 & 12 \\
    10 & [ 2, 1, 2, 0, 1, 0, 2, 1, 2, 0, 1, 0 ] &  6 & 12 \\
    14 & [ 1, 2, 1, 0, 1, 2, 1, 0, 1, 2, 1, 0 ] &  4 & 12 \\
    35 & [ 1, 1, 1, 1, 1, 1, 1, 1, 1, 1, 1, 1 ] &  1 & 12 \\
    68 & [ 1, 2, 1, 1, 1, 0, 0, 2, 0, 2, 2, 2 ] & 12 & 14 \\
    69 & [ 0, 1, 2, 0, 2, 1, 1, 2, 1, 2, 2, 2 ] & 12 & 16 \\
\end{array}
\]
Since $m[q,n,z] = 8$,
the upper bound for $\LL( A[q,n,z] )$ is $\frac{n(q-1)}{m[q,n,z]} + 1 = 4$.
If this bound would be attained then
the above vector with coefficient sum $16$ would be the sum of two cyclic
shifts of the vector with coefficient sum $8$.
However, no such decomposition is possible, and thus
$\LL( A[q,n,z] ) = 3$.
\end{example}

The following example shows nonisomorphic algebras $A[q,n,z]$
which have the same Loewy vector.

\begin{example}\label{same_Loewy_vector_nonisom}
(i)
  The algebras of smallest dimension with this property
  are $A[3,4,40]$ and $A[19,2,40]$.
  They have the Loewy vector $(1, 10, 19, 10, 1)$,
  and they can be distinguished as follows.
  View $A[q,n,z]$ as an algebra over the ring of integers,
  and let $A[q,n,z]_p$ denote its reduction modulo $p$.
  Then $U[q,n,z] = \{ x^2; x \in \J( A[q,n,z]_2 ) \}$ is a subspace
  of $A[q,n,z]_2$.
  We have $\dim( U[3,4,40] ) = 7$ and $\dim( U[19,2,40] ) = 11$.

(ii)
  The algebras $A[29,6,117]$ and $A[35,6,117]$ have
  the Loewy vector $(1,104,12,1)$,
  and the vector spaces $U[q,n,z]$ defined above have dimension $3$.

  We compute the cardinality of
  $\{ (x,y) \in A[q,n,z]_2; x y \in U[q,n,z] + \langle b_z \rangle \}$
  for the two parameter sets and get
  $2^{221} \cdot 119$ and $2^{216} \cdot 1069$, respectively.
\end{example}

Checking more examples, we get the following.

\begin{rem}\label{ref_database_examples}
Using {\GAP}~\cite{GAP},
we have computed the Loewy vectors of all $A[q,n,z]$ with
$1 \leq z \leq 10\,000$,
where $q$ runs over representatives of cyclic subgroups of the
group of prime residues modulo $z$ and $n = \ord_z(q)$.
\begin{itemize}
\item
    There are $768\,512$ such parameter pairs $(q,z)$,
    the number of pairwise different Loewy vectors is $475\,581$.
\item
    $\LL( A[q,n,z] ) = 3$ occurs in $191\,608$ cases,
    and Loewy vectors of the form $(1, k, 1, \ldots, 1)$
    and of length larger than $3$ occur in $37\,400$ cases;
    here the isomorphism type of $A[q,n,z]$ is determined by the
    dimension,
    by \cite[Proposition~\refPartI{isomorphism_by_Loewy_vector}]{BHHK1}.
\item
    Moreover, it happens in many cases that mapping corresponding basis
    vectors $b_i$ of two algebras with equal Loewy vector to each other
    defines an isomorphism.
    Checking for this special kind of isomorphism reduces the
    possible number of isomorphism types to $484\,234$.
    At this stage,
    we know that at most $7\,042$ Loewy vectors can belong to
    more than one isomorphism type.
\item
    Nonisomorphic algebras $A[q,n,z]$ with the same Loewy vector can occur,
    see Example~\ref{same_Loewy_vector_nonisom}.
\item
    Using the dimensions of the subspaces
    $V_{p,k} = \{ x \in \J(A[q,n,z]_p); x^{p^k} = 0 \}$ in the reduction
    modulo $p$ as invariants, we can show that several of the remaining
    candidates are nonisomorphic.
    At this stage,
    we know that the number of possible isomorphism types is at least
    $477\,912$ and at most $484\,234$.

    Using the dimensions of the ideals $V_{p,k} A[q,n,z]_p$,
    $J^i + S_j$, and $J^i S_j$,
    where $J= \J(A[q,n,z])$ and $S_j$ is the $j$-th member of the
    socle series of $A[q,n,z]$,
    yields a few more proofs of nonisomorphism.
    At this stage,
    we know that the number of possible isomorphism types is at least
    $478\,145$.
\item
    The smallest algebras $A[q,n,z]$ for which we currently do not know
    whether they are isomorphic are $A[11,6,171]$ and $A[68,6,171]$,
    they have the Loewy vector $(1,125,45,1)$.
\end{itemize}

The upper bound from
\cite[Theorem~\refPartI{thm_upper_and_lower_bounds_I}]{BHHK1}
is not attained for $10\,721$ parameter pairs;
some properties of these cases are listed below.
\begin{itemize}
\item
    The only examples of dimension up to $100$ are
    $A[3,12,70]$, $A[5,12,91]$, and $A[8,12,95]$.
\item
    The unique example for $z \leq 10\,000$
    where 
    $\LL( A[q,n,z] )$ is strictly smaller than
    $\lfloor \frac{n(q-1)}{m[q,n,z]} \rfloor$
    is $A[9,15,5\,551]$,
    where we have $m[q,n,z] = 24$ and $\LL( A[q,n,z] ) = 4$.
\item
    The unique example for $z \leq 10\,000$
    where the upper bound is not attained
    and $e = (q^n-1)/z$ is a prime power
    is $A[3,43,862]$,
    where $e = 380\,808\,546\,861\,411\,923$ is actually a prime.
    Note that $e$ divides $(q^n-1)/(q-1)$.
\item
    The example of smallest dimension
    with Loewy length at least $4$ is $A[7,12,195]$.
\item
    The smallest value of $n$ is $5$,
    it occurs in $13$ cases,
    the one of smallest dimension is $A[223,5,1\,353]$.
\item
    The smallest value of $e$ is $275$,
    it occurs exactly for $A[4,10,3\,813]$.
    Note that $e$ divides $(q^n-1)/(q-1)$.
\end{itemize}
Of course the chosen enumeration may be misleading,
since it is based on selecting certain parameter pairs.
However, this way we can get at least some measure how good the upper bound
from \cite[Theorem~\refPartI{thm_upper_and_lower_bounds_I}~(i)]{BHHK1} is.
\end{rem}


\begin{table}
\caption{$m(q,e)$ for $e \leq 30$}
\label{table_meq_small}
\begin{sideways}
\begin{minipage}{19.5cm}
\scriptsize
\[
\begin{array}{r|rrrrrrrrrrrrrrrrrrrrrrrrrrrrr|r}
 & 2 & 3 & 4 & 5 & 6 & 7 & 8 & 9 & 10 & 11 & 12 & 13 & 14 & 15 & 16 & 17 & 18 & 19 & 20 & 21 & 22 & 23 & 24 & 25 & 26 & 27 & 28 & 29 & 30 & \\
\hline
    1 & 2 & 3 & 4 & 5 & 6 & 7 & 8 & 9 & 10 & 11 & 12 & 13 & 14 & 15 & 16 & 17 & 18 & 19 & 20 & 21 & 22 & 23 & 24 & 25 & 26 & 27 & 28 & 29 & 30 & 1 \\
    2 &  & 2 &  & 2 &  & 3 &  & 2 &  & 2 &  & 2 &  & 4 &  & 2 &  & 2 &  & 3 &  & 3 &  & 2 &  & 2 &  & 2 &  & 2 \\
    3 &  &  & 2 & 2 &  & 2 & 4 &  & 2 & 3 &  & 3 & 2 &  & 4 & 2 &  & 2 & 4 &  & 4 & 3 &  & 2 & 6 &  & 2 & 2 &  & 3 \\
    4 &  &  &  & 2 &  & 3 &  & 3 &  & 3 &  & 2 &  & 6 &  & 2 &  & 3 &  & 3 &  & 3 &  & 2 &  & 3 &  & 2 &  & 4 \\
    5 &  &  &  &  & 2 & 2 & 4 & 2 &  & 3 & 4 & 2 & 2 &  & 4 & 2 & 2 & 3 &  & 2 & 4 & 2 & 8 &  & 2 & 2 & 4 & 2 &  & 5 \\
    6 &  &  &  &  &  & 2 &  &  &  & 2 &  & 2 &  &  &  & 2 &  & 3 &  &  &  & 3 &  & 5 &  &  &  & 2 &  & 6 \\
    7 &  &  &  &  &  &  & 2 & 3 & 2 & 2 & 6 & 2 &  & 3 & 4 & 2 & 6 & 3 & 4 &  & 2 & 2 & 6 & 2 & 2 & 3 &  & 4 & 6 & 7 \\
    8 &  &  &  &  &  &  &  & 2 &  & 2 &  & 2 &  & 4 &  & 2 &  & 2 &  & 7 &  & 3 &  & 2 &  & 2 &  & 2 &  & 8 \\
    9 &  &  &  &  &  &  &  &  & 2 & 3 &  & 3 & 4 &  & 8 & 2 &  & 3 & 4 &  & 4 & 3 &  & 2 & 6 &  & 4 & 2 &  & 9 \\
    10 &  &  &  &  &  &  &  &  &  & 2 &  & 2 &  &  &  & 2 &  & 2 &  & 3 &  & 2 &  &  &  & 9 &  & 2 &  & 10 \\
    11 &  &  &  &  &  &  &  &  &  &  & 2 & 2 & 4 & 5 & 4 & 2 & 2 & 3 & 10 & 3 &  & 2 & 4 & 5 & 2 & 2 & 4 & 2 & 10 & 11 \\
    12 &  &  &  &  &  &  &  &  &  &  &  & 2 &  &  &  & 2 &  & 2 &  &  &  & 3 &  & 2 &  &  &  & 2 &  & 12 \\
    13 &  &  &  &  &  &  &  &  &  &  &  &  & 2 & 3 & 4 & 2 & 6 & 2 & 4 & 6 & 2 & 3 & 12 & 2 &  & 3 & 4 & 2 & 6 & 13 \\
    14 &  &  &  &  &  &  &  &  &  &  &  &  &  & 2 &  & 2 &  & 2 &  &  &  & 2 &  & 2 &  & 2 &  & 2 &  & 14 \\
    15 &  &  &  &  &  &  &  &  &  &  &  &  &  &  & 2 & 2 &  & 2 &  &  & 4 & 2 &  &  & 2 &  & 14 & 2 &  & 15 \\
    16 &  &  &  &  &  &  &  &  &  &  &  &  &  &  &  & 2 &  & 3 &  & 3 &  & 3 &  & 5 &  & 3 &  & 4 &  & 16 \\
    17 &  &  &  &  &  &  &  &  &  &  &  &  &  &  &  &  & 2 & 3 & 4 & 2 & 2 & 2 & 8 & 2 & 2 & 2 & 4 & 2 & 4 & 17 \\
    18 &  &  &  &  &  &  &  &  &  &  &  &  &  &  &  &  &  & 2 &  &  &  & 3 &  & 2 &  &  &  & 2 &  & 18 \\
    19 &  &  &  &  &  &  &  &  &  &  &  &  &  &  &  &  &  &  & 2 & 3 & 2 & 2 & 6 & 2 & 2 & 9 & 2 & 2 & 6 & 19 \\
    20 &  &  &  &  &  &  &  &  &  &  &  &  &  &  &  &  &  &  &  & 2 &  & 2 &  &  &  & 2 &  & 4 &  & 20 \\
    21 &  &  &  &  &  &  &  &  &  &  &  &  &  &  &  &  &  &  &  &  & 2 & 2 &  & 5 & 2 &  &  & 2 &  & 21 \\
    22 &  &  &  &  &  &  &  &  &  &  &  &  &  &  &  &  &  &  &  &  &  & 2 &  & 2 &  & 3 &  & 2 &  & 22 \\
    23 &  &  &  &  &  &  &  &  &  &  &  &  &  &  &  &  &  &  &  &  &  &  & 2 & 2 & 2 & 2 & 4 & 4 & 4 & 23 \\
    24 &  &  &  &  &  &  &  &  &  &  &  &  &  &  &  &  &  &  &  &  &  &  &  & 2 &  &  &  & 4 &  & 24 \\
    25 &  &  &  &  &  &  &  &  &  &  &  &  &  &  &  &  &  &  &  &  &  &  &  &  & 2 & 3 & 4 & 4 &  & 25 \\
    26 &  &  &  &  &  &  &  &  &  &  &  &  &  &  &  &  &  &  &  &  &  &  &  &  &  & 2 &  & 2 &  & 26 \\
    27 &  &  &  &  &  &  &  &  &  &  &  &  &  &  &  &  &  &  &  &  &  &  &  &  &  &  & 2 & 2 &  & 27 \\
    28 &  &  &  &  &  &  &  &  &  &  &  &  &  &  &  &  &  &  &  &  &  &  &  &  &  &  &  & 2 &  & 28 \\
    29 &  &  &  &  &  &  &  &  &  &  &  &  &  &  &  &  &  &  &  &  &  &  &  &  &  &  &  &  & 2 & 29 \\
    30 &  &  &  &  &  &  &  &  &  &  &  &  &  &  &  &  &  &  &  &  &  &  &  &  &  &  &  &  &  & 30 \\
\end{array}
\]
\end{minipage}
\end{sideways}
\end{table}

\begin{table}
\caption{$m(q,e)$ for $31 \leq e \leq 60$}
\label{table_meq_middle}
\begin{minipage}{\textheight}
\scriptsize
$
\begin{array}{r|p{12cm}}
 e & \mbox{\rm $\mathbf{m}$ for sets of residues $\not= 1 \bmod e$} \\
\hline
31 &
\textbf{2}: \{ 3, 6, 11, 12, 13, 15, 17, 21, 22, 23, 24, 26, 27, 29, 30 \},
\textbf{3}: \{ 5, 7, 9, 10, 14, 18, 19, 20, 25, 28 \},
\textbf{5}: \{ 2, 4, 8, 16 \} \\
32 &
\textbf{2}: \{ 31 \},
\textbf{4}: \{ 3, 5, 7, 11, 13, 15, 19, 21, 23, 27, 29 \},
\textbf{8}: \{ 9, 25 \},
\textbf{16}: \{ 17 \} \\
33 &
\textbf{2}: \{ 2, 8, 17, 29, 32 \},
\textbf{3}: \{ 4, 5, 7, 13, 14, 16, 19, 20, 25, 26, 28, 31 \},
\textbf{6}: \{ 10 \},
\textbf{11}: \{ 23 \} \\
34 &
\textbf{2}: \{ 3, 5, 7, 9, 11, 13, 15, 19, 21, 23, 25, 27, 29, 31, 33 \} \\
35 &
\textbf{2}: \{ 19, 24, 34 \},
\textbf{3}: \{ 2, 3, 12, 17, 18, 23, 32, 33 \},
\textbf{4}: \{ 4, 9, 13, 27 \},
\textbf{5}: \{ 11, 16, 26, 31 \},
\textbf{7}: \{ 8, 22, 29 \},
\textbf{10}: \{ 6 \} \\
36 &
\textbf{2}: \{ 11, 23, 35 \},
\textbf{4}: \{ 5, 17, 29 \},
\textbf{6}: \{ 7, 31 \},
\textbf{12}: \{ 13, 25 \},
\textbf{18}: \{ 19 \} \\
37 &
\textbf{2}: \{ 2, 3, 4, 5, 6, 8, 11, 13, 14, 15, 17, 18, 19, 20, 21, 22, 23, 24, 25, 27, 28, 29, 30, 31, 32, 35, 36 \},
\textbf{3}: \{ 7, 9, 10, 12, 16, 26, 33, 34 \} \\
38 &
\textbf{2}: \{ 3, 13, 15, 21, 27, 29, 31, 33, 37 \},
\textbf{4}: \{ 5, 9, 17, 23, 25, 35 \},
\textbf{6}: \{ 7, 11 \} \\
39 &
\textbf{2}: \{ 17, 23, 38 \},
\textbf{3}: \{ 2, 4, 7, 10, 11, 16, 19, 20, 22, 28, 29, 32, 35, 37 \},
\textbf{4}: \{ 5, 8 \},
\textbf{6}: \{ 25, 31, 34 \},
\textbf{13}: \{ 14 \} \\
40 &
\textbf{2}: \{ 39 \},
\textbf{4}: \{ 3, 7, 13, 19, 23, 27, 37 \},
\textbf{8}: \{ 9, 17, 29, 33 \},
\textbf{10}: \{ 11, 31 \},
\textbf{20}: \{ 21 \} \\
41 &
\textbf{2}: \{ 2, 3, 4, 5, 6, 7, 8, 9, 11, 12, 13, 14, 15, 17, 19, 20, 21, 22, 23, 24, 25, 26, 27, 28, 29, 30, 31, 32, 33, 34, 35, 36, 38, 39, 40 \},
\textbf{5}: \{ 10, 16, 18, 37 \} \\
42 &
\textbf{2}: \{ 5, 17, 41 \},
\textbf{4}: \{ 11, 23 \},
\textbf{6}: \{ 13, 19, 25, 31, 37 \},
\textbf{14}: \{ 29 \} \\
43 &
\textbf{2}: \{ 2, 3, 5, 7, 8, 12, 18, 19, 20, 22, 26, 27, 28, 29, 30, 32, 33, 34, 37, 39, 42 \},
\textbf{3}: \{ 4, 6, 9, 10, 11, 13, 14, 15, 16, 17, 21, 23, 24, 25, 31, 35, 36, 38, 40, 41 \} \\
44 &
\textbf{2}: \{ 7, 19, 35, 39, 43 \},
\textbf{4}: \{ 3, 5, 9, 13, 15, 17, 21, 25, 27, 29, 31, 37, 41 \},
\textbf{22}: \{ 23 \} \\
45 &
\textbf{2}: \{ 14, 29, 44 \},
\textbf{3}: \{ 7, 13, 22, 43 \},
\textbf{4}: \{ 2, 8, 17, 23, 32, 38 \},
\textbf{5}: \{ 11, 41 \},
\textbf{6}: \{ 4, 34 \},
\textbf{9}: \{ 19, 28, 37 \},
\textbf{10}: \{ 26 \},
\textbf{15}: \{ 16, 31 \} \\
46 &
\textbf{2}: \{ 5, 7, 11, 15, 17, 19, 21, 33, 37, 43, 45 \},
\textbf{4}: \{ 3, 9, 13, 25, 27, 29, 31, 35, 39, 41 \} \\
47 &
\textbf{2}: \{ 5, 10, 11, 13, 15, 19, 20, 22, 23, 26, 29, 30, 31, 33, 35, 38, 39, 40, 41, 43, 44, 45, 46 \},
\textbf{3}: \{ 2, 3, 4, 6, 7, 8, 9, 12, 14, 16, 17, 18, 21, 24, 25, 27, 28, 32, 34, 36, 37, 42 \} \\
48 &
\textbf{2}: \{ 47 \},
\textbf{4}: \{ 11, 23, 35 \},
\textbf{6}: \{ 19, 31, 43 \},
\textbf{8}: \{ 5, 29, 41 \},
\textbf{12}: \{ 7, 13, 37 \},
\textbf{16}: \{ 17 \},
\textbf{24}: \{ 25 \} \\
49 &
\textbf{2}: \{ 3, 5, 6, 10, 12, 13, 17, 19, 20, 24, 26, 27, 31, 33, 34, 38, 40, 41, 45, 47, 48 \},
\textbf{3}: \{ 2, 4, 9, 11, 16, 18, 23, 25, 30, 32, 37, 39, 44, 46 \},
\textbf{7}: \{ 8, 15, 22, 29, 36, 43 \} \\
50 &
\textbf{2}: \{ 3, 7, 9, 13, 17, 19, 23, 27, 29, 33, 37, 39, 43, 47, 49 \},
\textbf{10}: \{ 11, 21, 31, 41 \} \\
51 &
\textbf{2}: \{ 50 \},
\textbf{3}: \{ 5, 7, 10, 11, 14, 19, 20, 22, 23, 25, 28, 29, 31, 37, 40, 41, 43, 44, 46, 49 \},
\textbf{4}: \{ 2, 8, 26, 32, 38, 47 \},
\textbf{6}: \{ 4, 13, 16 \},
\textbf{17}: \{ 35 \} \\
52 &
\textbf{2}: \{ 23, 43, 51 \},
\textbf{4}: \{ 5, 7, 11, 15, 17, 19, 21, 25, 31, 33, 37, 41, 45, 47, 49 \},
\textbf{6}: \{ 3, 35 \},
\textbf{8}: \{ 9, 29 \},
\textbf{26}: \{ 27 \} \\
53 &
\textbf{2}: \{ 2, 3, 4, 5, 6, 7, 8, 9, 11, 12, 14, 17, 18, 19, 20, 21, 22, 23, 25, 26, 27, 29, 30, 31, 32, 33, 34, 35, 37, 38, 39, 40, 41, 43, 45, 48, 50, 51, 52 \},
\textbf{3}: \{ 10, 13, 15, 16, 24, 28, 36, 42, 44, 46, 47, 49 \} \\
54 &
\textbf{2}: \{ 5, 11, 17, 23, 29, 35, 41, 47, 53 \},
\textbf{6}: \{ 7, 13, 25, 31, 43, 49 \},
\textbf{18}: \{ 19, 37 \} \\
55 &
\textbf{2}: \{ 19, 24, 29, 39, 54 \},
\textbf{3}: \{ 2, 3, 7, 8, 13, 17, 18, 27, 28, 37, 38, 42, 47, 48, 52, 53 \},
\textbf{4}: \{ 4, 9, 14, 32, 43, 49 \},
\textbf{5}: \{ 6, 16, 26, 31, 36, 41, 46, 51 \},
\textbf{10}: \{ 21 \},
\textbf{11}: \{ 12, 23, 34 \} \\
56 &
\textbf{2}: \{ 31, 47, 55 \},
\textbf{4}: \{ 3, 5, 11, 19, 23, 27, 37, 39, 45, 51, 53 \},
\textbf{8}: \{ 9, 13, 17, 25, 33, 41 \},
\textbf{14}: \{ 15, 43 \},
\textbf{28}: \{ 29 \} \\
57 &
\textbf{2}: \{ 2, 8, 14, 29, 32, 41, 50, 53, 56 \},
\textbf{3}: \{ 4, 5, 7, 10, 11, 13, 16, 17, 22, 23, 25, 26, 28, 31, 34, 35, 40, 43, 44, 46, 47, 49, 52, 55 \},
\textbf{6}: \{ 37 \},
\textbf{19}: \{ 20 \} \\
58 &
\textbf{2}: \{ 3, 5, 9, 11, 13, 15, 17, 19, 21, 27, 31, 33, 35, 37, 39, 41, 43, 47, 51, 55, 57 \},
\textbf{4}: \{ 7, 23, 25, 45, 49, 53 \} \\
59 &
\textbf{2}: \{ 2, 6, 8, 10, 11, 13, 14, 18, 23, 24, 30, 31, 32, 33, 34, 37, 38, 39, 40, 42, 43, 44, 47, 50, 52, 54, 55, 56, 58 \},
\textbf{3}: \{ 3, 4, 5, 7, 9, 12, 15, 16, 17, 19, 20, 21, 22, 25, 26, 27, 28, 29, 35, 36, 41, 45, 46, 48, 49, 51, 53, 57 \} \\
60 &
\textbf{2}: \{ 59 \},
\textbf{4}: \{ 17, 23, 29, 47, 53 \},
\textbf{6}: \{ 7, 19, 43 \},
\textbf{10}: \{ 11 \},
\textbf{12}: \{ 13, 37, 49 \},
\textbf{20}: \{ 41 \},
\textbf{30}: \{ 31 \} \\
\end{array}
$
\end{minipage}
\end{table}

\begin{table}
\caption{$m(q,e)$ for $61 \leq e \leq 100$}
\label{table_meq_large}
\begin{minipage}{20cm}
\scriptsize
$
\begin{array}{r|p{12cm}}
 e & \mbox{\rm $\mathbf{m}$ for smallest generators of residues $\not= 1 \pmod e$} \\
\hline
61 &
\textbf{2}: \{ 2, 3, 4, 8, 11, 14, 21, 60 \},
\textbf{3}: \{ 12, 13 \},
\textbf{4}: \{ 9 \} \\
62 &
\textbf{2}: \{ 3, 15, 37, 61 \},
\textbf{4}: \{ 7 \},
\textbf{6}: \{ 5, 33 \} \\
63 &
\textbf{2}: \{ 5, 17, 20, 47, 62 \},
\textbf{3}: \{ 31 \},
\textbf{4}: \{ 44 \},
\textbf{5}: \{ 11 \},
\textbf{6}: \{ 2, 13, 25, 40 \},
\textbf{7}: \{ 29 \},
\textbf{9}: \{ 4, 10, 37, 55 \},
\textbf{14}: \{ 8 \},
\textbf{21}: \{ 22 \} \\
64 &
\textbf{2}: \{ 63 \},
\textbf{4}: \{ 3, 5, 7, 15, 31 \},
\textbf{8}: \{ 9 \},
\textbf{16}: \{ 17 \},
\textbf{32}: \{ 33 \} \\
65 &
\textbf{2}: \{ 2, 4, 7, 8, 18, 64 \},
\textbf{3}: \{ 3 \},
\textbf{4}: \{ 12, 17, 19, 34 \},
\textbf{5}: \{ 6, 9, 16, 21, 36 \},
\textbf{10}: \{ 51 \},
\textbf{13}: \{ 14, 27 \} \\
66 &
\textbf{2}: \{ 17, 65 \},
\textbf{4}: \{ 5 \},
\textbf{6}: \{ 7, 25, 43 \},
\textbf{22}: \{ 23 \} \\
67 &
\textbf{2}: \{ 2, 3, 30, 66 \},
\textbf{3}: \{ 4, 29 \},
\textbf{4}: \{ 9 \} \\
68 &
\textbf{2}: \{ 67 \},
\textbf{4}: \{ 3, 5, 9, 13, 15, 33, 47 \},
\textbf{34}: \{ 35 \} \\
69 &
\textbf{2}: \{ 5, 68 \},
\textbf{3}: \{ 2, 4, 7 \},
\textbf{6}: \{ 22 \},
\textbf{23}: \{ 47 \} \\
70 &
\textbf{2}: \{ 19, 69 \},
\textbf{4}: \{ 3, 9, 13, 23 \},
\textbf{10}: \{ 11, 31, 41 \},
\textbf{14}: \{ 29, 43 \} \\
71 &
\textbf{2}: \{ 7, 14, 23, 70 \},
\textbf{3}: \{ 2 \},
\textbf{4}: \{ 20 \},
\textbf{5}: \{ 5 \} \\
72 &
\textbf{2}: \{ 23, 71 \},
\textbf{4}: \{ 11, 35 \},
\textbf{6}: \{ 7, 43 \},
\textbf{8}: \{ 5, 17, 41, 53 \},
\textbf{12}: \{ 13 \},
\textbf{18}: \{ 19, 55 \},
\textbf{24}: \{ 25 \},
\textbf{36}: \{ 37 \} \\
73 &
\textbf{2}: \{ 3, 5, 6, 7, 9, 10, 18, 27, 72 \},
\textbf{3}: \{ 2, 8 \} \\
74 &
\textbf{2}: \{ 3, 5, 11, 23, 31, 73 \},
\textbf{4}: \{ 7 \},
\textbf{6}: \{ 47 \} \\
75 &
\textbf{2}: \{ 14, 74 \},
\textbf{3}: \{ 13 \},
\textbf{4}: \{ 2, 32 \},
\textbf{5}: \{ 11 \},
\textbf{6}: \{ 4, 7, 49 \},
\textbf{15}: \{ 16 \},
\textbf{25}: \{ 26 \} \\
76 &
\textbf{2}: \{ 3, 27, 75 \},
\textbf{4}: \{ 5, 13, 23, 37, 65 \},
\textbf{6}: \{ 7 \},
\textbf{8}: \{ 45 \},
\textbf{38}: \{ 39 \} \\
77 &
\textbf{2}: \{ 6, 10, 17, 76 \},
\textbf{3}: \{ 2, 3, 4 \},
\textbf{4}: \{ 20, 32 \},
\textbf{7}: \{ 8, 15 \},
\textbf{11}: \{ 12, 23, 34 \},
\textbf{14}: \{ 43 \} \\
78 &
\textbf{2}: \{ 17, 77 \},
\textbf{4}: \{ 5, 11 \},
\textbf{6}: \{ 7, 25, 29, 31, 43, 55 \},
\textbf{26}: \{ 53 \} \\
79 &
\textbf{2}: \{ 3, 12, 24, 78 \},
\textbf{3}: \{ 2, 23 \},
\textbf{4}: \{ 8 \} \\
80 &
\textbf{2}: \{ 79 \},
\textbf{4}: \{ 7, 19, 39, 47, 53 \},
\textbf{6}: \{ 43 \},
\textbf{8}: \{ 3, 13, 29, 57 \},
\textbf{10}: \{ 11, 31, 71 \},
\textbf{16}: \{ 9, 17, 49 \},
\textbf{20}: \{ 21 \},
\textbf{40}: \{ 41 \} \\
81 &
\textbf{2}: \{ 2, 8, 26, 80 \},
\textbf{3}: \{ 4 \},
\textbf{9}: \{ 10 \},
\textbf{27}: \{ 28 \} \\
82 &
\textbf{2}: \{ 3, 5, 7, 9, 23, 81 \},
\textbf{6}: \{ 37 \} \\
83 &
\textbf{2}: \{ 2, 82 \},
\textbf{3}: \{ 3 \} \\
84 &
\textbf{2}: \{ 47, 83 \},
\textbf{4}: \{ 5, 11, 41, 53 \},
\textbf{6}: \{ 19, 55, 67 \},
\textbf{12}: \{ 13, 25, 61 \},
\textbf{14}: \{ 71 \},
\textbf{28}: \{ 29 \},
\textbf{42}: \{ 43 \} \\
85 &
\textbf{2}: \{ 13, 38, 84 \},
\textbf{3}: \{ 3, 42 \},
\textbf{4}: \{ 2, 4, 9, 12, 14, 33 \},
\textbf{5}: \{ 6, 21, 26 \},
\textbf{10}: \{ 16 \},
\textbf{17}: \{ 18, 69 \} \\
86 &
\textbf{2}: \{ 3, 7, 27, 85 \},
\textbf{4}: \{ 9 \},
\textbf{6}: \{ 11, 49 \} \\
87 &
\textbf{2}: \{ 5, 86 \},
\textbf{3}: \{ 2, 4, 10 \},
\textbf{4}: \{ 17, 20 \},
\textbf{6}: \{ 7, 28, 46 \},
\textbf{29}: \{ 59 \} \\
88 &
\textbf{2}: \{ 7, 87 \},
\textbf{4}: \{ 3, 5, 13, 15, 19, 43 \},
\textbf{8}: \{ 9, 17, 21, 65 \},
\textbf{22}: \{ 23, 67 \},
\textbf{44}: \{ 45 \} \\
89 &
\textbf{2}: \{ 3, 5, 11, 12, 34, 88 \},
\textbf{4}: \{ 2 \} \\
90 &
\textbf{2}: \{ 29, 89 \},
\textbf{4}: \{ 17, 23 \},
\textbf{6}: \{ 7, 49 \},
\textbf{10}: \{ 11, 71 \},
\textbf{18}: \{ 19, 37 \},
\textbf{30}: \{ 31 \} \\
91 &
\textbf{2}: \{ 10, 12, 17, 62, 90 \},
\textbf{3}: \{ 2, 3, 4, 9, 11, 16, 19, 30, 45, 68 \},
\textbf{4}: \{ 5, 6, 18, 25, 34 \},
\textbf{6}: \{ 48 \},
\textbf{7}: \{ 8, 15, 22, 36 \},
\textbf{13}: \{ 27, 40, 53 \},
\textbf{14}: \{ 64 \} \\
92 &
\textbf{2}: \{ 7, 91 \},
\textbf{4}: \{ 3, 5, 9, 45 \},
\textbf{46}: \{ 47 \} \\
93 &
\textbf{2}: \{ 11, 23, 26, 92 \},
\textbf{3}: \{ 5, 7, 13, 14, 25, 37, 46 \},
\textbf{5}: \{ 2 \},
\textbf{6}: \{ 4, 61 \},
\textbf{31}: \{ 32 \} \\
94 &
\textbf{2}: \{ 5, 93 \},
\textbf{4}: \{ 3 \} \\
95 &
\textbf{2}: \{ 14, 69, 94 \},
\textbf{3}: \{ 2, 7, 17 \},
\textbf{4}: \{ 4, 8, 18 \},
\textbf{5}: \{ 6, 21, 31 \},
\textbf{6}: \{ 49 \},
\textbf{10}: \{ 11, 56 \},
\textbf{19}: \{ 39, 58 \} \\
96 &
\textbf{2}: \{ 95 \},
\textbf{4}: \{ 11, 23, 47 \},
\textbf{6}: \{ 19, 31 \},
\textbf{8}: \{ 5, 41 \},
\textbf{12}: \{ 7, 13, 79 \},
\textbf{16}: \{ 17 \},
\textbf{24}: \{ 25 \},
\textbf{32}: \{ 65 \},
\textbf{48}: \{ 49 \} \\
97 &
\textbf{2}: \{ 2, 4, 5, 6, 8, 19, 22, 33, 36, 96 \},
\textbf{3}: \{ 35 \} \\
98 &
\textbf{2}: \{ 3, 13, 19, 97 \},
\textbf{4}: \{ 9 \},
\textbf{6}: \{ 67 \},
\textbf{14}: \{ 15 \} \\
99 &
\textbf{2}: \{ 2, 8, 32, 98 \},
\textbf{3}: \{ 4, 5, 7 \},
\textbf{4}: \{ 26 \},
\textbf{6}: \{ 43 \},
\textbf{9}: \{ 19, 37 \},
\textbf{11}: \{ 23, 89 \},
\textbf{18}: \{ 10 \},
\textbf{33}: \{ 34 \} \\
100 &
\textbf{2}: \{ 19, 99 \},
\textbf{4}: \{ 3, 7, 9, 13, 49, 57 \},
\textbf{10}: \{ 11 \},
\textbf{20}: \{ 21 \},
\textbf{50}: \{ 51 \} \\
\end{array}
$
\end{minipage}
\end{table}



\end{document}